\documentclass{cpamart1Modified}     
\received{Month 200X}       
\volume{000}
\startingpage{1}                      

\authorheadline{F. Monard, R. Nickl, G.P. Paternain}
\titleheadline{Consistent Inversion of Noisy Non-Abelian X-Ray Transforms}

\usepackage{amsmath,amscd}
\usepackage{amssymb}
\usepackage{amsthm}
\usepackage{graphicx}
\usepackage{mathrsfs}
\usepackage{hyperref}

\usepackage{mathptmx}

\newtheorem{Theorem}{Theorem}[section]
\newtheorem{Lemma}[Theorem]{Lemma}
\newtheorem{Corollary}[Theorem]{Corollary}

\theoremstyle{definition}
\newtheorem{Definition}[Theorem]{Definition}
\theoremstyle{remark}
\newtheorem{Remark}[Theorem]{Remark}
\newtheorem{Condition}[Theorem]{Condition}

\def\ba{\boldsymbol{a}}
\def \bb{\boldsymbol{b}}
\def \bP{\boldsymbol{P}}
\def \balpha{ {\boldsymbol{\alpha}} }

\def \<{\langle}
\def \>{\rangle}

\def \p{\partial}

\def \dim{{\mbox {dim}}\,}

\def\V{\mbox{Var}}

\def\R\re
\def\V{\bf V}

\def \Cm{{\mathbb C}}
\def \re{{\mathbb R}}
\def \mR{{\mathbb R}}

\def \C{{\mathbb C}}

\def \V{{\bf V}}

\def \O{{\mathcal O}}

\def \so{{\mathfrak s}{\mathfrak o}}
\def \u{{\mathfrak u}}

\newcommand{\norm}[1]{\lVert #1 \rVert}

\newcommand{\frob}[1]{\left| #1\right|_F}
\newcommand{\id}{\text{id}}

\begin{document}                        


\title{Consistent Inversion of Noisy Non-Abelian X-Ray Transforms}

\author{Fran\c{c}ois Monard}{Department of Mathematics, University of California, Santa Cruz, CA 95064}
\author{Richard Nickl}{Department of Pure Mathematics and Mathematical Statistics,
University of Cambridge,
Cambridge CB3 0WB, UK}
\author{Gabriel P. Paternain}{Department of Pure Mathematics and Mathematical Statistics,
University of Cambridge,
Cambridge CB3 0WB, UK}





\begin{abstract}For $M$ a simple surface, the non-linear statistical inverse problem of recovering a matrix field $\Phi: M \to  \so(n)$ from discrete, noisy measurements of the $SO(n)$-valued scattering data $C_\Phi$ of a solution of a matrix ODE is considered ($n\geq 2$). Injectivity of the map $\Phi \mapsto C_\Phi$ was established by [Paternain, Salo, Uhlmann; Geom.~Funct.~Anal. 2012, \cite{PSUGAFA}]. 

 A statistical algorithm for the solution of this inverse problem based on Gaussian process priors is proposed, and it is shown how it can be implemented by infinite-dimensional MCMC methods. It is further shown that as the number $N$ of measurements of point-evaluations of $C_\Phi$ increases, the statistical error in the recovery of $\Phi$ converges to zero in $L^2(M)$-distance at a rate that is algebraic in $1/N$, and approaches $1/\sqrt N$ for smooth matrix fields $\Phi$. The proof relies, among other things, on a new stability estimate for the inverse map $C_\Phi \to \Phi$.
 
Key applications of our results are discussed in the case $n=3$ to \textit{polarimetric neutron tomography}, see [Desai et al., Nature Sc.~Rep.~2018, \cite{DLSS}] and [Hilger et al., Nature Comm.~2018, \cite{Hetal}].

\end{abstract}

\maketitle   



\setcounter{tocdepth}{2}
 \tableofcontents


 \section{Introduction}
 
\subsection{Non-Abelian $X$-ray transforms} \label{sec:NAXRT}

Our object of study is the non-abelian $X$-ray transform, a mapping from a matrix-valued field $\Phi$ defined on a Riemannian surface with boundary $(M,g,\partial M)$, to its {\em scattering data} $C_\Phi$, defined at the influx boundary $\partial_+ SM$ of $M$, given by 
\begin{align*}
    \partial_+ SM = \{ (x,v) \in TM,\ x\in \partial M,\ g_x(v,v) = 1,\ \<v,\nu_x\>_g\leq 0 \}, 
\end{align*}
where $TM$ is the tangent bundle of $M$, and $\nu_x$ denotes the outward unit normal at $x\in \partial M$.

We will assume that the surface $M$ is {\em simple} in the sense that it is (topologically) a disk, it has no conjugate points, and a strictly convex boundary. Strictly convex domains in the plane (and small perturbations of them) are examples of simple surfaces. In this context, all unit-speed geodesics\footnote{Unit-speed geodesics are locally defined dynamically through the equation $\nabla_{\dot\gamma} \dot \gamma = 0$ with $\nabla$ the Levi-Civita connection, and satisfying $g_{\gamma(t)} (\dot \gamma(t), \dot\gamma(t)) = 1$ for all $t$ where $\gamma(t)$ is defined.} in $M$ exit $M$ in finite time. This fact allows us to identify $\partial_+ SM$ with the space of geodesics on $M$, by associating to any $(x,v)\in \partial_+ SM$ the unique geodesic $\gamma$ passing through $(x,v)$.

Let $\Phi:M \to \Cm^{n \times n}$ be a smooth map. Given  a unit-speed geodesic $\gamma:[0,T]\to M$ with endpoints $\gamma(0), \gamma(T)\in \partial M$, we may define the scattering data of $\Phi$ on $\gamma$ to be $C_\Phi(\gamma):= U(0)$, where $U:[0,T]\to \Cm^{n\times n}$ satisfies the linear system of ODE's
\begin{align*}
    \dot U + \Phi(\gamma(t)) U = 0, \qquad U(T) = \id. 
\end{align*}
This problem, backward in time for convention here, is well-posed and leads to a unique definition of $U(0)$, containing cumulated information about $\Phi$ along the geodesic $\gamma$. Note that when $\Phi$ is scalar, we obtain $\log U(0) = \int_0^T \Phi (\gamma(t))\ dt$, which is the classical X-ray/Radon transform of $\Phi$ along the curve $\gamma$. Considering the collection of all such data makes up the {\em scattering data} (or {\em non-Abelian X-ray transform}) of $\Phi$, viewed here as a map 
\begin{align*}
    C_\Phi\colon \partial_+ SM\to \Cm^{n\times n},
\end{align*}
and we are concerned with the problem of recovering $\Phi$ from $C_\Phi$. Inverting Abelian and non-Abelian X-ray transforms are examples of inverse problems in integral geometry, an active field permeating several tomographic imaging methods, see e.g. the recent topical review \cite{IlMo}. 

The problem of inverting the non-linear mapping $\Phi\mapsto C_\Phi$ in this generality has been recently solved in \cite{PS20}. Previous injectivity results were obtained, either by adding curvature conditions on the manifold, or by fixing a Lie group $G$ (realised as matrices, for simplicity) and its Lie algebra $\mathfrak{g}$, in turn asking whether a $\mathfrak{g}$-valued field $\Phi$ can be recovered from its $G$-valued scattering data $C_\Phi$. In this paper, we will mainly use the Lie groups $SO(n) = \{U\in \mathbb{R}^{n\times n},\ U^TU = \id,\;\det U = 1\}$, $U(n) = \{U\in \Cm^{n\times n},\ U^* U = \id \}$ and $SU(n) = U(n) \cap \{\det = 1\}$, and their Lie algebras $\so(n) = \{A\in \mathbb{R}^{n\times n},\ A^T + A = 0\}$, $\mathfrak{u}(n) = \{A\in \Cm^{n\times n},\ A^* + A = 0\}$ and $\mathfrak{su}(n) = \mathfrak{u}(n) \cap \{\text{tr} = 0\}$. Above, '$T$', '$*$', '$\det$' and '$\text{tr}$' refer to matrix 'transpose', 'conjugate transpose', 'determinant' and 'trace', respectively. Note the inclusions
\begin{align}
    SO(n)\subset SU(n)\subset U(n).
    \label{eq:incl}
\end{align}

The state of the art on this question can be written as follows: 
\begin{Theorem} Let $(M,g)$ be a simple surface. The map $\Phi\mapsto C_\Phi$ is injective in the following cases:

(a) $G=U(n)$ \cite{PSUGAFA};

(b) $G=GL(n,\mathbb{C})$ \cite{PS20}.

\label{thm:injective}
\end{Theorem}

The proof of (b) consists of a reduction to the unitary case in (a) via a factorization theorem in Loop Groups. Earlier injectivity results have been obtained by several authors, cf. \cite{E,No,Novikov_nonabelian} and references therein, particularly when $(M,g)$ is a domain in the Euclidean plane. 

\smallskip

The absence of concrete reconstruction formulas for the inverse map $C_\Phi \to \Phi$ when $n \ge 2$, and the challenge of dealing with physical experiments such as those arising in polarimetric neutron tomography (see Section \ref{PNT}), where $N$ discrete and noisy measurements $D_N \sim P_\Phi^N$ of $C_\Phi \in SO(3)$ are made (see Section \ref{obs} for details), motivate the main contribution of this article, which is to present a  statistical algorithm $\bar \Phi(D_N)$ that allows to recover $\Phi$. The implementation of $\bar \Phi(D_N)$ is detailed in Section \ref{impl}, and our main theoretical result is the statistical analogue of the injectivity result Theorem \ref{thm:injective}, namely the \textit{frequentist consistency of reconstruction in the large sample limit,} which somewhat informally can be stated as follows:
\begin{Theorem}
Suppose the data $D_N$ is generated from the probability distribution $P_{\Phi_0}^N$ where $\Phi_0: M \to \mathfrak{so}(n)$ is any smooth matrix field $\Phi_0$. Then we have that, as sample size $N \to \infty$, and in $P^N_{\Phi_0}$-probability, $$\|\bar \Phi(D_N) - \Phi_0\|_{L^2(M)} \to 0.$$ 
\end{Theorem}
See Theorem \ref{main} in Section \ref{bayes101} for a fully rigorous statement of this result, which in fact requires significantly weaker hypotheses on $\Phi_0$, and also specifies an explicit `algebraic' rate of convergence $N^{-\eta}$ in the last limit. 

The proof of the previous theorem relies on ideas from Bayesian nonparametric statistics \cite{vdVvZ08, GvdV17} and on new `quantitative versions' of the injectivity result in Theorem \ref{thm:injective} which are of independent interest and stated in Section \ref{theory}. 

\subsection{Polarimetric neutron tomography (PNT)} \label{PNT}

The basic problem in PNT consists in finding a magnetic field from spin measurements of neutrons \cite{Kardjilov_et_al2008, Dawson2009, DLSS, Hetal}. In this case the explicit relation is
\[ \Phi = \left[ \begin{matrix} 0 & B_{3}&-B_{2} \\- B_{3} & 0&B_{1}\\B_{2}&-B_{1}&0 \end{matrix} \right]\]
where $B=(B_{1},B_{2},B_{3})$ is the magnetic field. In the case of PNT one assumes that the underlying surface $M$ is just the disc in the plane (by slicing with 2D discs one can solve the 3D problem).

The details of the experiment of polarimetric neutron tomography may be found, e.g., in \cite{DLSS}. Here we give a description that is suitable for our purposes. The data produced by the experiment
is the orthogonal matrix $C^{-1}_{\Phi}(x,v)=C^{T}_{\Phi}(x,v)\in SO(3)$, where $C_{\Phi}(x,v)$ is the scattering data described above. The significance of this in terms of spin, is a follows: if a neutron travelling along the ray determined by $(x,v)$ enters the magnetic field with a spin $s_{in}\in \mathbb{S}^2$ ($\mathbb{S}^2$ denotes the Euclidean unit sphere in $\mathbb{R}^3$), it exits the field with spin $s_{out}=C^{-1}_{\Phi}(x,v)s_{in}\in \mathbb{S}^2$ (for an ensemble of polarized neutrons in a magnetic field it can be shown that they behave like a particle with a classical magnetic moment). The magnetic field $B$ is defined in 3D space, but the experiment makes measurements on a 2D plane and produces a global reconstruction by slicing. The geometry of the experiment is thus a 2D parallel beam geometry which is easily converted into fan-beam geometry as considered above.
The question is then how to manipulate the spin to produce the orthogonal matrix. This is done with an ingenious sequence of spin flippers and rotators placed before and after the magnetic field being measured. The material containing the magnetic field  can also be rotated so as to produce parallel beams from different angles.
After the spin has been manipulated it goes through an analyser; this device is essentially a spin filter that
only lets those neutrons with vertically aligned spin go through. The neutron count is then measured with a detector that produces an intensity reading. The spin of the entering beam is perfectly aligned with the spin of the analyser, so that the intensity measurement is actually a measurement of the angle of rotation of the spin due to the magnetic field. The key relation is given by \cite[Equation 1]{Kardjilov_et_al2008}
\begin{equation}
I=I_{0}A\frac{1}{2}(1+\cos\varphi),
\label{eq:savior}
\end{equation}
where $A$ is the attenuation of the medium, $I_{0}$ is the intensity of the incoming beam
and $\varphi$ is the angle by which the spin has rotated.

The use of the spin flipper allows the measurement of
\[I'=I_{0}A\frac{1}{2}(1-\cos\varphi),\]
and from this one deduces that
\[\cos\varphi=\frac{I-I'}{I+I'}\]
which then becomes an entry of our matrix $C^{T}_{\Phi}(x,v)$. 
By rotating by $\pi/2$ and flipping (rotation by $\pi$) one can thus produce the entire orthogonal matrix as data. 
In other words, if $\{e_{1},e_{2},e_{3}\}$ is the canonical basis of 3-space, $\cos\varphi$ gives
$C^{T}_{\Phi}(x,v)e_{i}\cdotp e_{j}$ for all $i,j$ and hence all the entries.
In some situations, where the attenuation of the medium is known, the use of spin flippers is not necessary and can be calibrated out.
Assuming an additive Gaussian noise in the intensities $I$, equation \eqref{eq:savior} approximately produces an additive Gaussian noise in the entries of the matrix $C_{\Phi}$ which is precisely the noise model we adopt below.

As in the articles \cite{DLSS,DLSS2} our approach reconstructs 3D magnetic fields of arbitrary direction and distribution. This provides a method able to investigate samples without imposing any a priori knowledge of the magnetic field orientation, and requires understanding of the full non-linear inverse problem. The recent preprint \cite{DLSS2} introduces a modified Newton-Kantorovich type algorithm for the solution of the non-linear problem, a Newton-type algorithm where the inversion of the Jacobian at each iteration only uses the differential of the map $\Phi \mapsto C_\Phi$ at the base point $\Phi_0 \equiv 0$. 

As pointed out in \cite{DLSS2}, the algorithm appears to work well for  small  enough  fields  (or large enough velocities of neutrons), but may fail due to ``phase wrapping" when the field is large enough. Our approach does not exhibit this problem.

\subsection{The statistical observation scheme}\label{obs}

Consider a simple surface $M$ as above with influx boundary $\partial_+ SM$, and a matrix valued map $$\Phi: M \to \mathfrak g$$ and scattering data $$C_\Phi: \partial_+ SM \to G.$$ Here we take $G = SO(n)$ for some $n\ge 2$, with corresponding Lie algebra $\mathfrak g = \so(n)$, the set of skew-symmetric matrices. Recall that in the key application to PNT from the previous subsection, $M$ is the flat disk and $n=3$. We could take $G=SU(n)$ and $\mathfrak g = \mathfrak{su}(n)$ just as well, but for sake of conciseness prefer to avoid a complex-valued statistical noise model in what follows.

To describe the statistical observation setting, let $\lambda$ be the uniform distribution (volume element) on $\partial_+ SM$ (see  (\ref{lambo}) below for a precise definition), and consider `design' random variables $$(X_i, V_i)_{i=1}^N \sim^{i.i.d.} \lambda \text{ on } \partial_+ SM.$$ These draws represent a randomised choice of the geodesics for which experiments are  performed -- they have to be `equally spaced' throughout `geodesic space' $\partial_+SM$ in a statistical sense. For each resulting measurement of $C_\Phi((X_i, V_i))$  the statistical observational error arising in the experiment is modelled by independent Gaussian matrix noise. More precisely let 
\begin{align*}
    (\varepsilon_{i,j,k}: 1 \le j,k \le n)_{i=1}^N\qquad \text{ be i.i.d. }\quad N(0,\sigma^2), ~\sigma>0,
\end{align*}
random variables that are independent of the $(X_i, V_i)$'s, and let $\mathcal E_i=(\varepsilon_{i,j,k})$ be the random $n \times n$ noise matrix which adds a Gaussian noise variable in each matrix entry to $C_\Phi((X_i, V_i))$. Our observations then consist of the sequence of $N$ random $n \times n$ matrices
\begin{equation}\label{model}
Y_i=(Y_{i,j,k}), Y_{i,j,k} = C_\Phi ((X_i, V_i))_{j,k} + \varepsilon_{i,j,k}, ~~ i=1, \dots, N; 1 \le j,k \le n.
\end{equation}
The variables $Y_{i,j,k}$ are all independent, and even i.i.d. for $j,k$ fixed. Conditionally on $(X_i, V_i)=(x_i,v_i)$ they are multivariate normal random variables with diagonal covariance and (vectorised) mean $C_\phi(x_i,v_i)_{j,k}$. Note that while $C_\Phi(x,y)$ takes values in $SO(n)$, the $Y_i$ are not in $SO(n)$ (or even $U(n)$) as we have not constrained  $\mathcal E_i$ at all -- this is in line with the physical experiments for PNT described in Section \ref{PNT} where statistical errors arise from noisy measurements of each matrix entry of $C_\Phi(x,v)$.  For the theory we will assume  that the noise variance $\sigma^2>0$ is fixed and known -- in practice it can be replaced by the estimated sample variance of the $Y_{i,j,k}$'s.

\smallskip

To fix notation: The joint law of the random variables $(Y_{i}, (X_i, V_i))_{i=1}^N$ in (\ref{model}) on $(\mathbb R^{n \times n} \times \partial_+SM)^N$ will be denoted by $P_\Phi^N = \times_{i=1}^N P_\Phi^i$, where we note $P_\Phi^i=P_\Phi^1$ for all $i$. We also write $P^N_\varepsilon$ for the law of the $(\mathcal E_i)_{i=1}^N$'s, $\lambda^N$ for the law of the $(X_i, V_i)_{i=1}^N$ and 
\begin{equation}\label{data}
D_N = \{Y_1, \dots, Y_N, (X_1, V_1), \dots, (X_N, V_N)\}
\end{equation}
 for the full data vector. The corresponding expectation operators are obtained by replacing `$P$' by `$E$' in the preceding expressions. The dependence on $\sigma^2$ will be suppressed in the notation. 

\subsection{Some geometric background and basic notation} \label{sec:notation}

We conclude this section by introducing some more basic notation that will be used throughout. 

Our background geometry is a {\em simple} surface with boundary $(M,g, \partial M)$. By 'simple', we mean (i) $M$ is non-trapping (in the sense that every maximal geodesic in $M$ has finite length), (ii) $M$ has no conjugate points and (iii) $\partial M$ is strictly convex (i.e. $\partial M$ has positive definite second fundamental form). We denote by $SM$ the unit tangent bundle of $M$, namely 
\begin{align*}
    SM = \{ (x,v)\in TM,\ g_x(v,v) = 1 \}.
\end{align*}
Its boundary $\partial SM:=\{(x,v)\in SM:\;x\in\partial M\}$ can be split into 'influx' and 'outflux' boundary, depending on whether the tangent vector points inside or outside, namely we define, for $\nu_{x}$ is the outer unit normal at $x\in\partial M$, 
\begin{align*}
    \p_{\pm}SM:=\{(x,v)\in \p SM : \pm\<v,\nu_x\>_g\leq 0\}.
\end{align*}

The manifolds $M$, $\partial M$, $SM$ and $\partial_{+}SM$ all carry natural volume elements, allowing us to define $L^2$ spaces below. Specifically, the Riemannian metric $g$ induces an area form $dx$ on $M$ and restricts to a metric on $\partial M$. The unit sphere bundle $SM$ carries the volume element $d\Sigma^3=dx\,dv$ where $dv$ is the length element in the unit circle $S_{x}\subset T_{x}M$. Finally the boundary $\partial SM$ of $SM$ carries the area form $d\Sigma^2 = ds\ dv$ where $dv$ is as above and $ds$ is the arclength (w.r.t. the metric $g$) along the boundary. Its restriction to $\partial_+ SM$ will be denoted by
\begin{equation}\label{lambo}
\lambda \equiv \frac{1}{Area(\partial_{+}SM)}d\Sigma^2|_{\partial_{+}SM}.
\end{equation}
The spaces $\Cm^n$ and $\Cm^{n\times n}$ will be equipped with the canonical Hermitian inner product $\langle\cdot, \cdot\rangle$ and induced norm $|\cdot|$. For elements in $\Cm^{n\times n}$, this corresponds to the Frobenius norm $\frob{A}^2:= \text{tr} (A^*A) = \sum_{i,j=1}^n |A_{i,j}|^2$, which is $U(n)$-invariant in the sense that for any $U\in U(n)$ and $A$ arbitrary, $\frob{AU} = \frob{UA} = \frob{A}$.

Given $(N,h)$ a $d$-dimensional Riemannian manifold (either $M$, $\partial M$, $SM$, $\partial_+ SM$, or $\partial SM$ as explained above), one may adapt the usual function spaces to $\Cm^n$- or $\Cm^{n\times n}$-valued functions as follows: $L^2(N,\Cm^{n\times n})$, $L^\infty(N,\Cm^{n\times n})$ with norms
\begin{align*}
    \|U\|^2_{L^2} := \int_N \frob{U}^2\ d\text{Vol}_h, \qquad \|U\|_{L^\infty} := \sup_{y\in N} \frob{U(y)}.
\end{align*}
One may differentiate functions using partial derivatives $\{\partial_{y_j}\}_{j=1}^d$ in coordinate charts, or equivalently, using $\{T_j\}_{j=1}^d$ a global basis of smooth vector fields on $N$ which pairwise commutes (it will be useful to adopt the latter viewpoint in later sections). 
Given a $d$-index $\balpha = (\alpha_1, \dots, \alpha_d)$, one may define $|\balpha| = \alpha_1 + \dots + \alpha_d$ and $T^{\balpha} = T_1^{\alpha_1}\cdots T_d^{\alpha_d}$. The metric $h$ equips $N$ with a distance function $d_h(x,y)$, and for $\beta\ge 0$, we can thus define H\"older spaces $C^\beta(N,\Cm^{n\times n})$ with norm 
\begin{align*}
    \|U\|_{C^\beta} = \sum_{|\balpha|\le \lfloor \beta\rfloor} \sup_{y\in N} \frob{T^{\balpha} U(y)} + \sum_{|\balpha| = \lfloor \beta\rfloor} \sup_{x\ne y\in N} \frac{\frob{T^{\balpha} U(x) - T^{\balpha}U(y)}}{d_h(x,y)^{\beta-\lfloor \beta\rfloor}},
\end{align*} 
with the second term removed when $\beta$ is an integer. We will also use $L^2$-based Sobolev spaces $H^s (N,\Cm^{n\times n})$ with norm 
\begin{align*}
    \|U\|^2_{H^s} = \sum_{|\balpha| \le s} \|T^{\balpha} U\|_{L^2}^2,
\end{align*}
for $s \in \mathbb N$, and defined by interpolation otherwise (see, e.g., \cite[Ch. 4]{T}).

As above, when clear from the context, the domain and/or codomain will be dropped from the notation. In the following sections, spaces of functions with codomain $SO(n)$, $SU(n)$ or their Lie algebras will make use of the same topology of the corresponding spaces of $\Cm^{n\times n}$-valued functions. The $c$-subscript attached to a space of maps defined on $M$ denotes the linear subspace of those maps that vanish identically outside of a compact subset of the interior $M^{int}$ of $M$.

\section{Theoretical results for the deterministic inverse problem}\label{theory}

When discrete measurements of the forward data $C_\Phi$ are corrupted by statistical noise, the injectivity result Theorem \ref{thm:injective} is not useful to reconstruct $\Phi$ from the observations, and we will discuss in the next section how to develop statistical methods that consistently solve this statistical inverse problem. The proofs that substantiate these methods are based on \textit{quantitative} versions of Theorem \ref{thm:injective} -- stability estimates -- as well as continuity properties of the forward map, and we describe in this section the analytical results we obtain. 

\smallskip

The results to follow hold when the codomain of the matrix fields is the largest of the three compact Lie groups introduced before Theorem \ref{thm:injective}, namely $U(n)$ (with Lie algebra $\u(n)$), see Eq. \eqref{eq:incl}.  

\begin{Theorem} \label{thm:stability}
    Let $(M,g)$ be a simple surface. Given two matrix fields $\Phi$ and $\Psi$ in $C^{1}(M,\u(n))$ there exists a constant $c(\Phi,\Psi)$ such that
    \begin{align*}
	\norm{\Phi-\Psi}_{L^{2}(M)}\leq c(\Phi,\Psi)\norm{C_{\Phi}C_{\Psi}^{-1}-\text{\rm id}}_{H^{1}(\partial_{+}SM)},	
    \end{align*}
    where $c(\Phi,\Psi)$ is a continuous function of $\|\Phi\|_{C^1} \vee \|\Psi\|_{C^1}$, explicitly 
    \begin{align}
	c(\Phi,\Psi)=C_1 (1+ (\norm{\Phi}_{C^1}\vee \norm{\Psi}_{C^1}))\ e^{C_2 (\norm{\Phi}_{C^1} \vee \norm{\Psi}_{C^1})},     
	\label{eq:cPhiPsi}
    \end{align}	
    and where the constants $C_1, C_2$ only depend on $(M,g)$.
\end{Theorem}


The proof of Theorem \ref{thm:stability} initially follows the approach for obtaining $L^{2}\to H^{1}$ stability estimates for the geodesic X-ray transform $I$ as presented  in \cite[Theorem 3.4.3]{Sharafut99}. Our starting point is the pseudo-linearisation formula 
\[ C_{\Phi}C_{\Psi}^{-1}= \id+I_{\Theta(\Phi,\Psi)}(\Phi-\Psi)\]
where $I_{\Theta(\Phi,\Psi)}$ is a geodesic X-ray transform with suitable weights, see Lemma \ref{lem:pseudolin}. To prove Theorem \ref{thm:stability} it suffices to show that
\[\norm{\Phi-\Psi}_{L^{2}(M)}\leq c(\Phi,\Psi)\norm{I_{\Theta(\Phi,\Psi)}(\Psi-\Phi)}_{H^{1}(\partial_{+}SM)}.\]
To this end, we use the energy identity (Pestov Identity) developed in \cite{PSUGAFA} for matrix weights arising for connections and matrix fields. The presence of the weights produces additional terms in the identity that need to be controlled to obain the estimate above and this is where most of the work lies.  The main idea for controlling them comes from \cite{PSUGAFA} where a connection with the right curvature is artificially introduced to control these terms. The connection is later removed by using (scalar) {\it holomorphic integrating factors} whose existence is guaranteed by the microlocal properties of the normal operator associated to the geodesic X-ray transform acting on functions.
Taming these integrating factors has a cost which is reflected in the constant $c(\Phi,\Psi)$ given in \eqref{eq:cPhiPsi}.

For the proof of Theorem \ref{main} below we also require `forward'  estimates in Sobolev and H\"older scales. These are less sophisticated in nature than the stability estimate above, and hold under less restrictive assumptions. Recall that $(M,g)$ is said to be {\it non-trapping} if there is no geodesic with infinite length (any simple manifold is non-trapping). 

\begin{Theorem} \label{thm:forward}
    Let $(M,g)$ be a non-trapping surface with strictly convex boundary. For any integer $k\ge 0$ and for every $\Phi, \Psi\in C^k(M,\mathfrak{u}(n))$, the following continuity estimates hold: 
    \begin{align}
	\|C_\Phi - C_\Psi \|_{H^k(\partial_+ SM,\C^{n\times n})} &\lesssim (1+ \norm{\Phi}_{C^k} + \norm{\Psi}_{C^k})^k \|\Phi-\Psi\|_{H^k(M, \C^{n\times n})}, \label{eq:cont0} \\ 
	\|C_\Phi - C_\Psi \|_{C^k(\partial_+ SM,\C^{n\times n})} &\lesssim (1+ \norm{\Phi}_{C^k} + \norm{\Psi}_{C^k})^k \|\Phi-\Psi\|_{C^k(M, \C^{n\times n})}, \label{eq:continf}
    \end{align}
    where by $\lesssim$ we mean that the inequality holds with some constant that only depends on $M$, $g$ and $k$. 
\end{Theorem}

In fact in the proof of Theorem \ref{main} we shall use instead of Theorem \ref{thm:stability} the following corollary of the previous two results:

\begin{Corollary}\label{cor:stab2}
    Under the same hypotheses as in Theorem \ref{thm:stability} and $c(\Phi,\Psi)$ as in \eqref{eq:cPhiPsi}, then
    \begin{equation} \label{stabi}
	\|\Phi - \Psi\|_{L^2(M)} \le C'\ c(\Phi,\Psi) (1+\norm{\Psi}_{C^1}) \norm{C_{\Phi}-C_{\Psi}}_{H^{1}(\partial_+ SM)},
    \end{equation}
    where $C'$ is independent of $\Phi$ or $\Psi$.
\end{Corollary}

\section{Bayesian inversion of non-Abelian $X$-ray transforms} \label{bayes101}

\subsection{Main results}

The main goal of this section is to introduce a method to infer the matrix field $\Phi \in C(M,\so(n))$ from discrete observations $D_N$ of the scattering data $C_\Phi$ described in Section \ref{obs}. We follow the general paradigm of \textit{Bayesian inverse problems} advocated by A. Stuart \cite{S10, DS16} which is also related to the paradigm of Bayesian numerical analysis \cite{D88, G19} in the noiseless case ($\sigma=0$). The idea is to start from a \textit{Gaussian process prior} $\Pi$ for the parameter $\Phi$ and to use Bayes' theorem to infer the best \textit{posterior guess} for $\Phi$ given data $D_N$.

\smallskip

We will state a theorem that shows that the posterior mean fields $\bar \Phi_N=E^\Pi[\Phi|D_N]$ corresponding to a flexible class of Lie-algebra valued Gaussian process priors $\Pi$ for $\Phi$ consistently recover the `true' $\Phi_0$ in the frequentist large sample limit as $N \to \infty$, when noisy experiments have been performed under $P^N_{\Phi_0}$ in the model (\ref{model}). In fact we will provide a stochastic convergence rate to zero of the recovery error that is algebraic in inverse sample size $1/N$. 

\smallskip

The proof of Theorem \ref{main} below provides a template to establish rigorous statistical guarantees for the Bayesian approach to other non-linear inverse problems as well. See Section \ref{prfbayes} and Remark \ref{genrem} for more discussion.

\smallskip

We emphasise that obtaining probabilistic consistency under $P_{\Phi_0}^N$ entails approximate uniformity of the design $(X_i, V_i)$ and rules out `adversarial' designs. Fixed (non-random) design $(x_i, v_i)$ that is sufficiently `equally spaced' throughout $\partial_+ SM$ could be considered as well in the theory that follows, either via appealing to asymptotic statistical equivalence results in nonparametric regression \cite{R08} or by tracking the numerical discretisation error explicitly through all the proofs that follow. For the purposes of the present paper we opt for the random design setting as it allows for a cleaner, unified probabilistic treatment of the measurement process.

\smallskip

To introduce the Bayesian approach more concisely, consider a prior $\Pi$ for a vector field $(B_1, \dots, B_{\bar n})$ by prescribing a  Borel probability measure on the space $\times_{j=1}^{\bar n} C(M)$ where $$\bar n = \frac{n (n-1)}{2}=\dim(\so(n)).$$ The natural isomorphism between $\times_{j=1}^{\bar n} C(M)$ and the space $C(M, \so(n))$ of continuous functions from $M$ to $\so(n)$ in turn generates a prior $\Pi$ for $\Phi$ by forming a $\so(n)$-valued field from the $B_i$'s. For instance in the case $n=3$ so that also $\bar n =3$, relevant in PNT, we construct $\Pi$ from
\begin{align}\label{matr}
\Phi(x) = \left[ \begin{matrix} 0 & B_{3}(x)&-B_{2}(x) \\- B_{3}(x) & 0&B_{1}(x)\\B_{2}(x)&-B_{1}(x)&0 \end{matrix} \right], ~x \in M.
\end{align}
Then we make the Bayesian model assumption that $$(Y_{i}, (X_i, V_i))_{i=1}^N|\Phi \sim P_\Phi^N~~\text{on }(\mathbb R^{n \times n} \times \partial_+SM)^N $$ which by Bayes' rule generates a conditional posterior distribution of $\Phi|(Y_{i}, (X_i, V_i))_{i=1}^N$ on $C(M, \so(n))$ --  it will be denoted by $\Pi(\cdot|(Y_{i}, (X_i, V_i))_{i=1}^N) \equiv \Pi(\cdot|D_N)$. The posterior distribution arises from a dominated family of probability measures (see (\ref{moddens}) below) and is hence given by
\begin{equation}\label{post}
\Pi(A|D_N)\equiv \Pi(A|Y_1, \dots, Y_N, (X_1, V_1), \dots, (X_N, V_N)) = \frac{\int_A e^{\ell_N(\Phi)}d\Pi(\Phi)}{\int_{} e^{\ell_N(\Phi)}d\Pi(\Phi)},
\end{equation}
for any Borel set $A$ in $C(M, \so(n))$. Here
\begin{equation} \label{likpo}
\ell_N(\Phi) = \sum_{i \le N}\ell_i(\Phi),~\text{ where }\ell_i(\Phi)  =  -\frac{1}{2\sigma^2}\sum_{1 \le j,k \le n}\big[Y_{i,j,k} -C_\Phi((X_i,V_i))_{j,k} \big]^2,
\end{equation}
is, up to additive constants, the \textit{log-likelihood function} of the observations. 
\smallskip

While what precedes was not specific to the choice of a particular prior, the main theorem to follow will hold for priors arising from certain $\so(n)$-valued \textit{Gaussian processes}. These will be constructed from a Gaussian base prior $\Pi'$ from which the coordinates $B_j$ of $\times_{j=1}^{\bar n} C(M)$ will be drawn independently. In fact we will require draws from $\Pi'$ to have $\beta$-H\"older continuous sample paths on $M$ almost surely. We refer, e.g., to \cite[Sections 2.1 and 2.6]{GN16} for the basic definitions of Gaussian measures and processes and their reproducing kernel Hilbert spaces (RKHS).
\begin{Condition}\label{pco}
For $\beta>0$ and $\alpha>\beta+1$, let $\Pi'$ be a centred Gaussian Borel probability measure on the Banach space $C(M)$ that is supported in a separable (measurable) linear subspace of $C^\beta(M)$, and assume its RKHS $(\mathcal H, \|\cdot\|_{\mathcal H})$ is continuously imbedded into the Sobolev space $H^\alpha(M)$. 
\end{Condition}
See Remark \ref{pcoex} for concrete examples and constructions of such Gaussian process priors with `maximal choice' $\mathcal H=H^\alpha(M)$ and arbitrary $\alpha>\beta+1$. 

Now given a random draw $f' \sim \Pi'$ we define a new random function 
  \begin{equation} \label{tampering}
  B(x) = B_N(x) =  \frac{f'(x)}{\sqrt{N^{1/(\alpha+1)}}},~x \in M, ~ f' \sim \Pi',
  \end{equation}
   and denote its law in $C(M)$ by $\Pi_B=\Pi_{B,N}$. Then let $B_1, \dots, B_{\bar n}$ be random functions on $M$ drawn as i.i.d.~copies from $\Pi_B$, and let the prior $\Pi= \times_{j=1}^{\bar n} \Pi_B$ for $\Phi$ be the resulting centred Gaussian product probability measure in the space $C(M, \so(n))\simeq \times_{j=1}^{\bar n} C(M)$ (see (\ref{matr}) for $n=3$). Shrinking the prior towards the origin in a $N$-dependent way as in (\ref{tampering}) is crucial in our proofs, see Remark \ref{scaling} for discussion.

The following theorem gives a bound for the convergence rate of the posterior mean 
\begin{equation} \label{pmean}
    \bar \Phi_N= \bar \Phi((Y_{i}, (X_i, V_i))_{i=1}^N)= E^{\Pi}[\Phi |(Y_{i}, (X_i, V_i))_{i=1}^N]
\end{equation}
towards the true field $\Phi_0$ in $L^2(M)$-loss, under the law $P_{\Phi_0}^N$ of the observations. Note that this mean (expected value) is understood in the usual sense of Bochner integrals and hence $\bar \Phi$ takes values in $C(M, \so(n))$ -- for fixed data vector $Y_{i}, (X_i, V_i)$ and since for $C_\Phi \in SO(n)$ the norms $\|C_\Phi\|_{L^\infty}$ are bounded by a fixed constant, this expected value exists almost surely by (\ref{post}) and a basic application of Fernique's theorem (see \cite[Exercise 2.1.5]{GN16}).  Let us say $\Phi \in \mathcal H$ if all matrix entries of $\Phi$ are contained in $\mathcal H$.
   
 \begin{Theorem}\label{main}
Suppose the Gaussian prior $\Pi$ for $\Phi$ arises as after (\ref{tampering}) with base prior $\Pi'$ satisfying Condition \ref{pco} for $\alpha>\beta+1, \beta>2$. Let $\bar \Phi_N$ be the mean (\ref{pmean}) of the posterior distribution $\Pi(\cdot|(Y_{i}, (X_i, V_i))_{i=1}^N)$ arising from observations (\ref{model}). Assume $\Phi_0 \in C^\alpha(M, \so(n)) \cap \mathcal H$. Then we have, for some $\eta>0$
$$P_{\Phi_0}^N \Big(\|\bar \Phi_N - \Phi_0\|_{L^2(M)} >N^{-\eta} \Big) \to 0~~\text{as } N \to \infty.$$ 
 \end{Theorem}

The proof is given in Section \ref{prfbayes}. We note that the constraint $\beta>2$ (and hence $\alpha>3$) could be relaxed to $\beta>1$ (and hence $\alpha>2$) at the expense of more technical proofs (see Remark \ref{beta}). We further remark that in the proof we  establish in particular that the random posterior measure $\Pi(\cdot|(Y_{i}, (X_i, V_i))_{i=1}^N)$ on $C(M, \so(n))$ concentrates with probability approaching one in a $N^{-\eta}$-diameter $L^2(M)$-ball centred at $\Phi_0$, see Theorem \ref{overall}.  
\subsection{Remarks and discussion}

  \begin{Remark}\normalfont [The exponent $\eta$.] In the proof (see (\ref{eta})) we show that $$\eta < \frac{\alpha}{(2\alpha+2)}\frac{\bar \beta-1}{\bar \beta},~~\text{ any integer } \bar \beta \text{ s.t. } 1<\bar \beta < \beta,$$ is permitted in the previous theorem. If $\Phi_0 \in C^\infty(M) = \cap_{\alpha>0} H^\alpha(M)$ and if we take priors $\Pi$ which verify Condition \ref{pco} for large enough $\alpha, \beta$ and $\mathcal H= H^\alpha(M)$ (possible by Remark \ref{pcoex}), then we can make $\eta$ as close to $1/2$ as desired, and it is easy to show that $\eta=1/2$ cannot be improved upon by any algorithm. So at least for smooth $\Phi_0$ the recovery guarantee from Theorem \ref{main} is (near-) optimal.  In the `low regularity case' where $\alpha$ is not large, our bound for $\eta$ may not be optimal. A conjecture for the optimal value for $\eta$ can be obtained from the much simpler linear and Abelian case ($n=1$) corresponding to the classical Radon transform, which is treated in \cite[Example 2.5]{NvdGW18}, where the exponent $\eta=\alpha/(2\alpha+3)$ is attained, which can be shown to be optimal in this special case.
 \end{Remark}

\begin{Remark}\normalfont [Construction of Gaussian priors.] \label{pcoex}
We describe here some Gaussian process priors verifying Condition \ref{pco} with $\mathcal H = H^\alpha(M)$.

\smallskip

As a first basic example consider the case where $M$ equals the unit disk $D=\{(x_1,x_2) \in \mathbb R^2: x_1^2+x_2^2 \le 1\}$ in $\mathbb R^2$ with `flat' (Euclidean) geometry, relevant in PNT.  For arbitrary $\alpha>0$ we can then take for $\Pi'$  the restriction to $D$ of a stationary Gaussian process on $\mathbb R^2$ with appropriate (Whittle-) Mat\'ern covariance function $k_\alpha$ (see \cite[p.313]{GvdV17} and Section \ref{impl} below). This gives a Gaussian prior on $C(D)$ with RKHS $\mathcal H$ equal to the space of restrictions to $D$ of elements of $H^\alpha(\mathbb R^2)$ (using Ex.2.6.5 in \cite{GN16}). This space is well known (e.g., \cite{T}, Ch.4) to co-incide with $H^\alpha(D)$, and the sample paths of this process lie in the separable subspace $C^{\beta_0}(D)$ of $C^\beta(D)$ for any $\beta<\beta_0<\alpha-1$, see \cite[p.575f]{GvdV17} for a proof. 

The preceding construction works for any smooth bounded domain $D$ in $\mathbb R^2$. In particular a simple surface $M$ is diffeo-morphic to a disc and the Sobolev spaces $H^\alpha(D)$ and $H^\alpha(M)$ co-incide with equivalent norms -- the Mat\'ern prior can thus be used even when $M$ equals $D$ equipped with a different Riemannian metric. Alternatively one can embed $M$ isometrically into a larger closed compact (boundary-less) manifold $S$ and use the orthonormal basis of eigenfunctions $\{e_k\}$ of the Laplace-Beltrami operator on $S$ to generate Gaussian random series $f_S(x)= \sum_k \sigma_k g_k e_{k} (x)$, $g_k \sim^{i.i.d.} N(0,1),~~x \in S,$ which after restriction to $M$ and for suitable choice of $\sigma_k>0$, generate Gaussian priors $\Pi$ with any prescribed Sobolev space $H^\alpha(M)$ as RKHS. 
\end{Remark}
 
 \begin{Remark}\label{scaling}\normalfont [Rescaled Gaussian Priors.]  While the use of Gaussian process techniques \cite{B75, F75, LL99} in the proof of Theorem \ref{main} is inspired by previous work in \cite{vdVvZ08, vdVvZ07} and also \cite{GN11} for `direct' problems, the inverse setting poses several challenges, particularly in the non-linear case. In our proofs we show how these challenges can be overcome by shrinking common Gaussian process priors towards the origin as in (\ref{tampering}) -- the shrinkage enforces the necessary additional `a-priori' regularisation of the posterior distribution to permit the use of our stability estimates. While similar re-scaled priors have been shown to work in some `direct' settings before (they appear as special cases of the rescaled priors studied in \cite{vdVvZ07}, see their Theorem 3.2), in our setting they play a crucial role: Without re-scaling the exponential growth in the $C^1$-norms of $\Phi$ of the constant (\ref{eq:cPhiPsi}) would render our stability estimate useless in the proofs.
 \end{Remark}
 
 \begin{Remark}\label{genrem} \normalfont [Related literature on Bayesian non-linear inverse problems.]
The study of statistical guarantees for the Bayesian approach to non-linear inverse problems has seen a recent surge of interest. In the references \cite{V13, NS17, N1} non-linear inverse problems of elliptic and parabolic type are studied. The results therein however only hold for specific `uniformly bounded wavelet' type priors -- while these are useful to develop a first theoretical understanding of Bayesian inversion algorithms, they posit very strong a priori assumptions on the parameter of interest and the efficient computability of the resulting posterior distribution is also unclear. 

The recent reference \cite{NvdGW18} obtains convergence rate results for optimisation based MAP-estimates (see Section \ref{sec:MCMC} for a brief discussion of those) in a general class of non-linear inverse problems. For non-linear forward maps as the ones relevant here, these MAP-estimates can be difficult to compute, and at any rate may behave quite differently from the posterior mean: The algorithm $E^{\Pi}[\Phi |(Y_{i}, (X_i, V_i))_{i=1}^N]$ studied here is a Bochner integral with respect to an infinite-dimensional and non-Gaussian posterior distribution and variational ideas from optimisation cannot be used directly in its analysis. In the proof of Theorem \ref{main} we develop new techniques that allow to prove convergence rates for such algorithms -- see Section \ref{prfbayes} for a discussion of the key ideas which are relevant in other settings, too. Indeed, the very recent references \cite{AN20, GN20} have already succeeded in adapting our proof template to other nonlinear inverse problems. For instance \cite{AN20} study statistical versions of a conceptually related boundary value problem arising with electrical impedance tomography (`Calder\'on problems'). Our results imply that statistical inversion of non-Abelian $X$-ray transforms (for `smooth parameters' $\Phi$) admits better (i.e., polynomial) convergence rates than the necessarily logarithmic (in inverse noise level) recovery guarantees derived in \cite{AN20} for the Cald\'eron problem (with smooth conductivities). \end{Remark}
 
 \begin{Remark}\normalfont [Towards Uncertainty Quantification.]
Theorem \ref{main} also serves as a starting point to prove more refined Bernstein-von Mises theorems that entail that the posterior distribution is approximated in a suitable infinite-dimensional space by a canonical Gaussian measure (cf.~\cite{CN13, CN14}). For a non-linear elliptic inverse problem a first result of this kind was recently proved in \cite{N1}, and for the linearisation of the non-linear problem considered here, such results were obtained in \cite{MNP}. In principle, joining the ideas of \cite{N1, MNP} with the techniques of the present paper, one can conjecture that Bernstein-von Mises theorems should also hold true for the case of non-Abelian $X$-ray transforms -- this is the subject of ongoing research.
  \end{Remark}

\section{Implementation of the algorithm}\label{impl}

In this section, we present some numerical reconstructions of an $\mathfrak{su}(2)$-valued matrix field $\Phi$ from its noisy scattering data $C_\Phi\in SU(2)$. In this case, $\Phi$ is generated by three real-valued components $B_1, B_2, B_3$, through the relation $\Phi = B_1\ \sigma_1 + B_2\ \sigma_2 + B_3\ \sigma_3$, where we have defined for basis of $\mathfrak{su}(2)$ 
\begin{align*}
    \sigma_1 = \frac{1}{2} \left[
	\begin{array}{cc}
	    i & 0 \\ 0 & -i
	\end{array}
    \right], \qquad 
    \sigma_2 = \frac{1}{2} \left[
	\begin{array}{cc}
	    0 & 1 \\ -1 & 0
	\end{array}
    \right], \qquad
    \sigma_3 = \frac{1}{2} \left[
	\begin{array}{cc}
	    0 & i \\ i & 0
	\end{array}
    \right],
\end{align*}
with structure equations $[\sigma_1,\sigma_2] = \sigma_3$, $[\sigma_2,\sigma_3] = \sigma_1$ and $[\sigma_3,\sigma_1] = \sigma_2$. The approach presented easily adapts to any $\so(n)$-, $\mathfrak{su}(n)$- or $\u(n)$-valued field (including the $\so(3)$-valued case of polarimetric neutron tomography, a close cousin of the present case), with some minor Lie group specific modifications to be made for an accurate computation of forward data.  

\subsection{Numerical domain and forward operator} \label{sec:numerics_forward}

The computational domain is an unstructured triangular mesh discretising the unit disk $M = \{x^2 + y^2 \le 1\}$ made of $N_v$ vertices, and functions on it are piecewise linear, uniquely determined by their values at the vertices. In particular $\Phi$ is regarded as an element of $\mathbb{R}^{3N_v}$.

\begin{figure}[htpb]
    \centering
    \includegraphics[trim=30 0 45 0, clip, height=0.15\textheight]{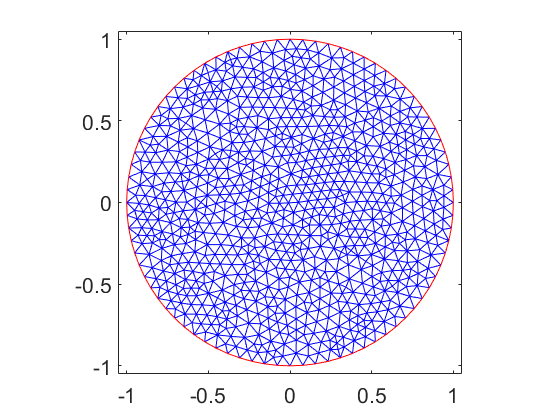}
    \includegraphics[trim=5 0 10 0, clip, height=0.15\textheight]{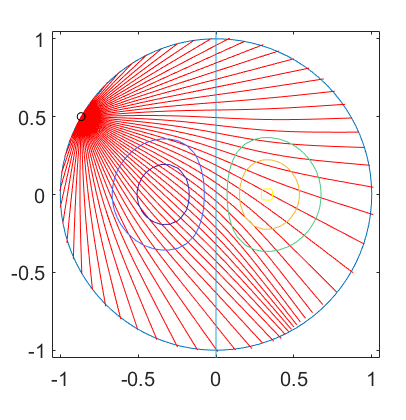}
    \includegraphics[trim=0 0 10 0, clip, height=0.15\textheight]{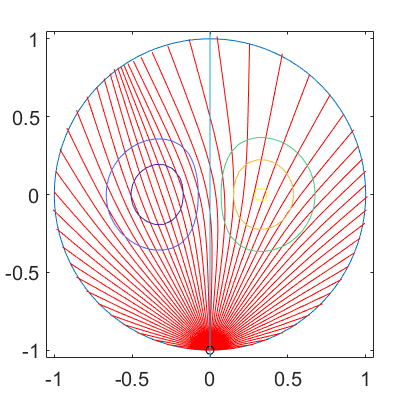}
    \includegraphics[trim=20 0 30 0, clip, height=0.15\textheight]{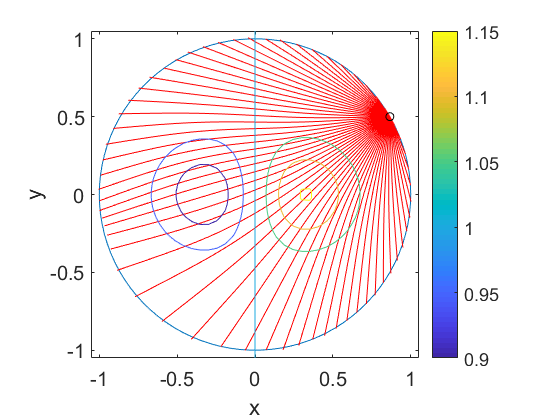}
    \caption{Left to right: An example of mesh with $N_v = 886$ vertices. Some geodesics for the metric we use in the examples that follow. A contour plot of the 'sound speed' $c = e^{-\bar \lambda}$ is superimposed.}
    \label{fig:mesh}
\end{figure}

The metric is isotropic, written as $g = e^{2\bar \lambda(x,y)} \id$, with scalar function $\bar \lambda$ given by
\begin{align*}
    \bar \lambda(x,y) = 0.3 (e^{-((x+0.3)^2 + y^2)/2\tau^2} - e^{-((x-0.3)^2 + y^2)/2\tau^2}), \qquad \tau = 0.25.
\end{align*}
Such an example can be seen to be non-trapping, have no conjugate points and a strictly convex boundary, i.e. $(M,g)$ is {\em simple}. The case of Euclidean geometry would correspond to $\bar \lambda\equiv 0$.
Geodesic (data) space, modelled as $\partial_+ SM$ is parameterised in fan-beam coordinates $(\beta,\alpha)\in (0,2\pi)\times (-\pi/2,\pi/2)$ (with uniform probability measure $d\lambda=d\beta\ d\alpha/(2\pi^2)$). 

Below we will draw $N$ geodesics uniformly at random, characterised by $N$ initial conditions $(\alpha_i,\beta_i)\in \partial_+ SM$, $1\le i\le N$, and our statistical algorithm will require numerical evaluation of the forward data $C_\Phi(\alpha_i,\beta_i)$ which we now describe: Out of each data point $(\alpha_i,\beta_i)$, we first compute a geodesic using a forward scheme with stepsize $h$ to solve a discretisation of the system 
\begin{align*}
    \dot x(t) = e^{-\bar \lambda} \cos \theta, \qquad \dot y(t) = e^{-\bar \lambda} \sin \theta, \qquad \dot \theta(t) = e^{-\bar \lambda} (-\sin\theta \partial_x \bar \lambda + \cos\theta \partial_y\bar \lambda),
\end{align*}
with initial condition $x(0) = \cos \beta_i$, $y(0) = \sin \beta_i$ and $\theta(0) = \beta_i+\pi+\alpha_i$, until the geodesic exits the domain. This produces a discretised geodesic
\begin{align*}
    \gamma_i = \{(x_i(t_j), y_i(t_j)), \quad t_j = jh, \quad 0\le j\le J_i\}.
\end{align*}

Once such a geodesic is computed, we must then discretise the matrix ODE
\begin{align*}
    \dot U (\gamma_i(t),\dot\gamma_i(t)) + \Phi (\gamma_i(t)) U(\gamma_i(t),\dot\gamma_i(t)) = 0, \qquad U(\gamma_i(0)) = \id.
\end{align*}
(The problem here is forward in time unlike that given in the introduction, though since $\Phi$ is $\u(n)$-valued, this amounts to computing the conjugate transpose of $C_\Phi$, which leads to the same problem.)

To discretise the above ODE, we denote $U^{(i,j)} := U(\gamma_i(t_j), \dot \gamma_i(t_j))$ and implement the scheme
\begin{align}
    U^{(i,j)} = \exp(-h \Phi^{(i,j-1)}) \cdot U^{(i,j-1)},\qquad 1\le j\le J_i,
    \label{eq:Uij}
\end{align}
where we have defined $\Phi^{(i,j-1)} = \Phi (x_i(t_{j-1}), y_i(t_{j-1}))$. In fact the code implements a predictor-corrector variant of this scheme for improved accuracy on the computation of the exponentials. 

The use of matrix exponentials in \eqref{eq:Uij} (compared to standard forward-marching schemes) ensures that the matrix solution $U$ numerically remains in $SU(2)$, and the computation of these exponentials can be done via an explicit formula, namely: for $A = a\ \sigma_1 + b\ \sigma_2 + c\ \sigma_3$ and denoting $|a| := \sqrt{a^2 + b^2 + c^2}$, we have for $l \in \mathbb R $
\begin{align*}
    \exp(l A) = \cos\left(\frac{l|a|}{2}\right) \id + \text{sinc } \left(\frac{l|a|}{2}\right) l A, \qquad (\text{sinc }x := (\sin x)/x)
\end{align*}
(Note that the formula above would need to be adapted if a Lie algebra $\mathfrak g$ different from $\mathfrak {su}(2)$ is of interest.) The evaluation of $\Phi^{(i,j-1)}$ is done by barycentric combination of the values of $\Phi$ at the three vertices of the triangle containing $(x_i(t_{j-1}), y_i(t_{j-1}))$. 

After implementing \eqref{eq:Uij}, the scattering data $C_\Phi (\gamma_i)$ is nothing but $U^{(i,J_i)}$ (in fact, the other values $U^{(i,j)}$ for $j<J_i$ are not kept in memory after computation). The magnetic field $\Phi$ we will use in the experiments below as well as its noiseless scattering data $C_\Phi$ are visualised Fig. \ref{fig:PhiCPhi}.

\begin{figure}[htpb]
    \centering    
    \includegraphics[height=0.15\textheight]{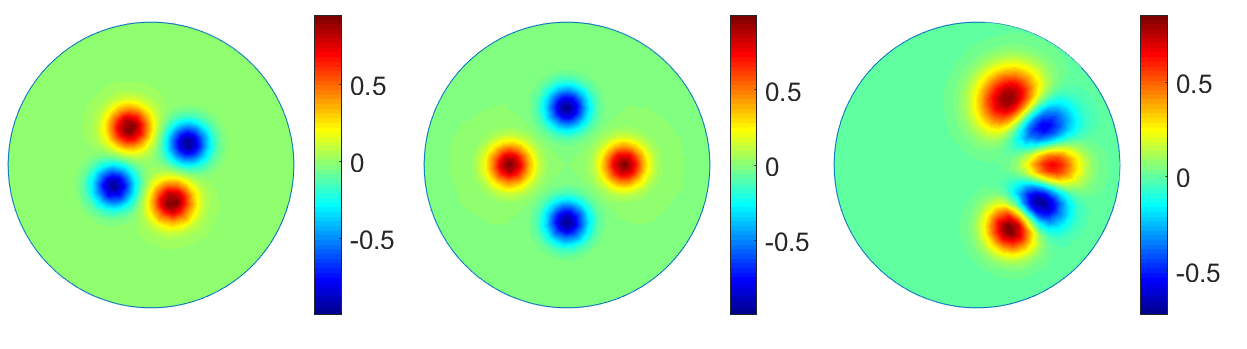} \\
    \includegraphics[width=\textwidth]{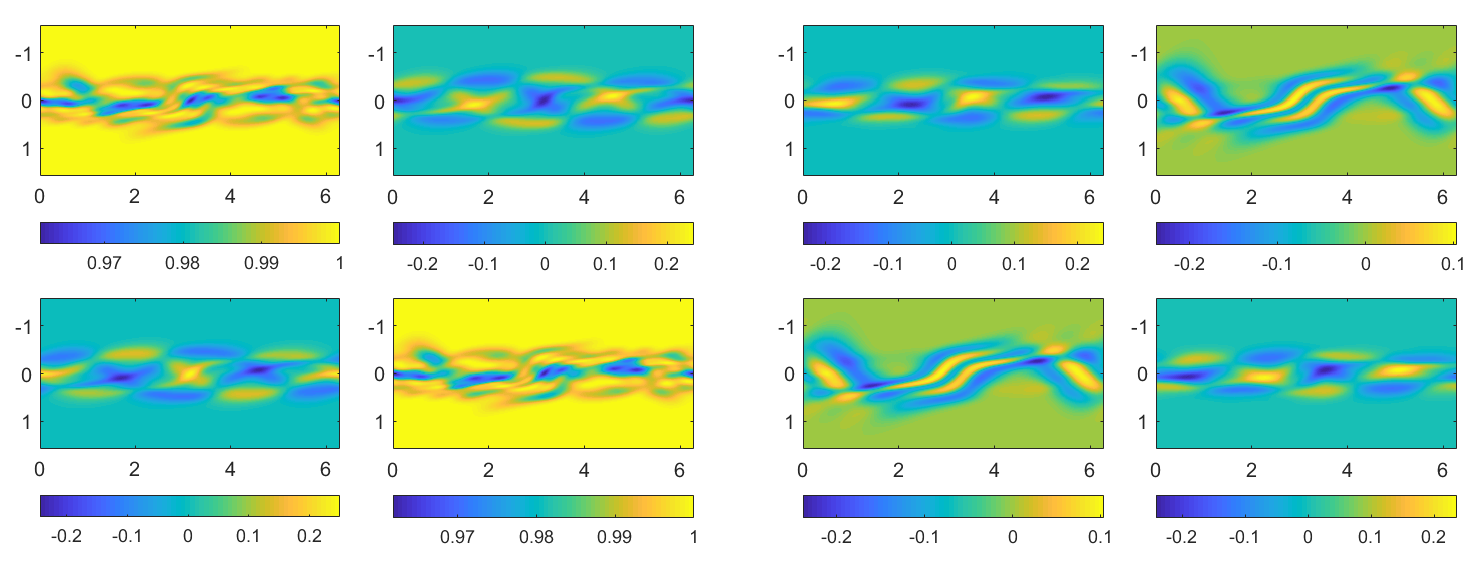}
    \caption{Top: the three components $(B_1, B_2, B_3)$ of the magnetic field realised as $\Phi_0 = B_1\ \sigma_1 + B_2\ \sigma_2 + B_3\ \sigma_3$. Bottom: real (left $2\times 2$ block) and imaginary (right $2\times 2$ block) parts of the scattering data $C_{\Phi_0}:\partial_+ SM\to SU(2)$ for the magnetic field $\Phi_0$ visualised on top.}
    \label{fig:PhiCPhi}
\end{figure}

As we will use Monte-Carlo Markov Chains (MCMC) in the following section, let us mention that once the mesh is fixed, some computations are done prior to the MCMC, namely, all geodesics as well as the triangle indices and barycentric weights along them.

\subsection{Statistical estimation through MCMC} \label{sec:MCMC}

Given data as in (\ref{model}), a common approach to inverse problems would be to compute a Tikhonov regulariser which minimises a penalised least squares fit functional (with, e.g., Sobolev-norm penalty)
\begin{equation}\label{tikh}
    Q_N(\Phi) = \frac{1}{2\sigma^2} \sum_{i=1}^N \frob{Y_i - C_{\Phi}(X_i,V_i)}^2 + \frac{1}{2}\|\Phi\|_{H^\alpha}^2
\end{equation}
over the space of all matrix fields $\Phi: M \to \mathfrak g$ where $\mathfrak g$ is the Lie algebra describing the constraint on the co-domain of $\Phi$. The map $Q_N$ is not convex, and efficient computation of the global minimiser may be challenging. One approach would be to use a gradient based iterative scheme \cite{KNS08} but the algorithmic stability of these (or other variational) methods is unclear in the setting considered here.   

The optimiser of the functional (\ref{tikh}) can be shown to correspond to a posterior mode, or `maximum a posteriori estimate (MAP)', of a Gaussian process prior $\Pi$ on $C(M, \mathfrak g)$ with RKHS equal to $H^\alpha$ (see \cite{DLSV13} for a general result of this kind). Instead of computing that maximiser, one may compute other posterior characteristics such as the posterior mean (average) $E^\Pi[\Phi|D_N]=E^\Pi[\Phi|(Y_{i}, (X_i, V_i))_{i=1}^N]$, which in our non-linear setting is different from the MAP estimate. 

 For Gaussian priors, MCMC algorithms such as the preconditioned Crank-Nicolson (pCN) method (see \cite{CRSW13}) are available to sample from the posterior distribution. To introduce the algorithm, note that as in \eqref{likpo}, the log-likelihood function given the data $(Y_i,(X_i,V_i))_{i=1}^N$ equals, up to additive constants,
\begin{align*}
    \ell(\Phi) := - \frac{1}{2\sigma^2}\sum_{i=1}^N \frob{Y_i - C_{\Phi}(X_i,V_i)}^2.
\end{align*}
One then approximates the posterior mean $E^\Pi[\Phi|(Y_i,(X_i,V_i))_{i=1}^N]$ by a Monte Carlo average  $\widehat{\Phi} = \frac{1}{N_{s}} \sum_{n=0}^{N_s}\Phi_n$ of a Markov chain $(\Phi_n)$ of length $N_s$ as follows: 

\smallskip

Let $\Pi$ be a Gaussian prior for $\Phi$; initialise $\Phi_n = 0$ for $n=0$, then repeat:  
\begin{enumerate}
    \item Draw $\Psi\sim \Pi$ and for $\delta>0$ define the proposal $p_{\Phi_n} := \sqrt{1-2\delta}\ \Phi_n + \sqrt{2\delta}\ \Psi$.   
    
    \smallskip
      
    \item Set 
	\begin{align*}
	    \Phi_{n+1} = \left\{
	    \begin{array}{ll}
		p_{\Phi_n}, & \text{with probability } 1\wedge \exp (\ell(p_{\Phi_n})-\ell(\Phi_n)), \\
		\Phi_n, & \text{otherwise}.
	    \end{array}
	    \right.
	\end{align*}
\end{enumerate}

The algorithm is terminated at $n=N_s$ and requires evaluation of $\ell (\Phi_n)$ and thus of the scattering data $C_{\Phi_n}(X_i, V_i)$ for every $\Phi_n$ and $(X_i, V_i)$. For $\mathfrak g=\mathfrak{su}(2)$ relevant in the simulations that follow, this can be done as described in Section \ref{sec:numerics_forward}.

\smallskip

The invariant measure of the Markov chain $\{\Phi_n\}$ equals the posterior distribution $\Pi(\cdot|D_N)$, and under certain conditions that are compatible with our setting, \cite{HSV14} derived dimension-free spectral gaps which imply that the distribution of $\Phi_n$ mixes rapidly towards $\Pi(\cdot|D_N)$. The approximation of $E^\Pi[\Phi|D_N]$ by $\widehat{\Phi} = \frac{1}{N_{s}} \sum_{n=0}^{N_s}\Phi_n$ can thus be expected to compare to the one of the standard central limit theorem, with corresponding non-asymptotic error guarantees, see Section 4 in \cite{HSV14}.

\smallskip

To perform numerical simulations, we discretise $\Phi= \sum_{i=1}^3 B_i \sigma_i: M \to \mathfrak {su}(2)$ as in Section \ref{sec:numerics_forward} and for each $B_i$ choose an independent Mat\'ern prior (cf.~Remark \ref{pcoex}) with parameters $(\nu,\ell)$, which on functions on the mesh (i.e., vectors in $\mathbb{R}^{N_v}$) uses the covariance matrix $C_{i,j} = k_{\nu,\ell}(|x_i-x_j|)$ for $1\le i,j\le N_v$, with positive definite kernel
\begin{align*}
    k_{\nu,\ell}(r) := \frac{2^{1-\nu}}{\Gamma(\nu)} \left( \frac{\sqrt{2\nu}r}{\ell} \right)^\nu K_\nu \left( \frac{\sqrt{2\nu}r}{\ell} \right), 
\end{align*}
with $K_\nu$ the modified Bessel function of the second kind. The constant $\nu$ controls the Sobolev regularity while $\ell$ controls the characteristic lengthscale of the samples.

\begin{figure}[htpb]
    \centering
    \boxed{\includegraphics[trim=60 20 40 10, clip, width=0.15\textwidth]{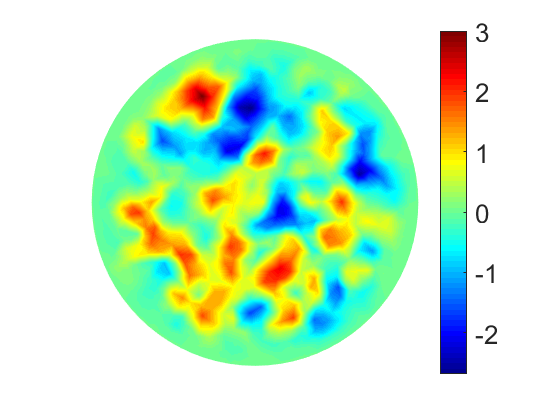}
    \includegraphics[trim=60 20 40 10, clip, width=0.15\textwidth]{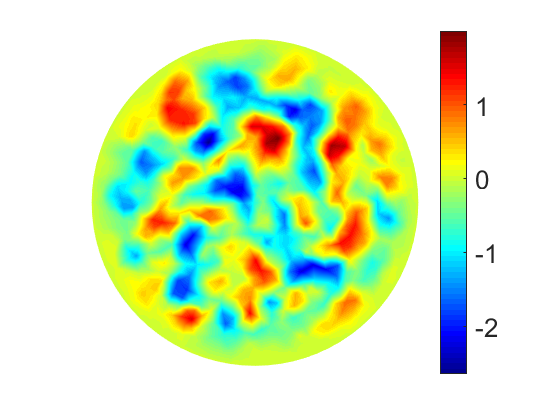}}
    \boxed{\includegraphics[trim=60 20 40 10, clip, width=0.15\textwidth]{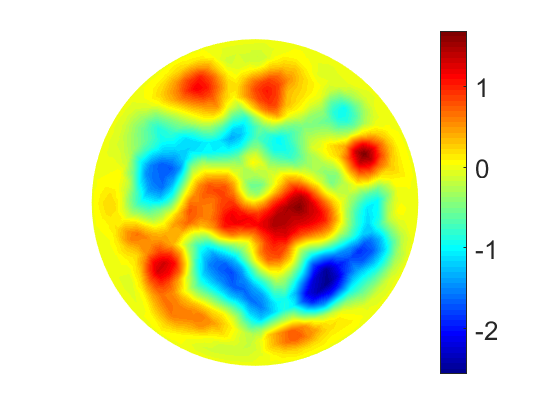}
    \includegraphics[trim=60 20 40 10, clip, width=0.15\textwidth]{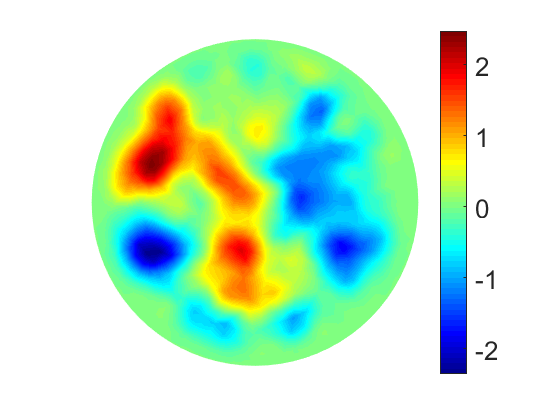}}
    \boxed{\includegraphics[trim=60 20 40 10, clip, width=0.15\textwidth]{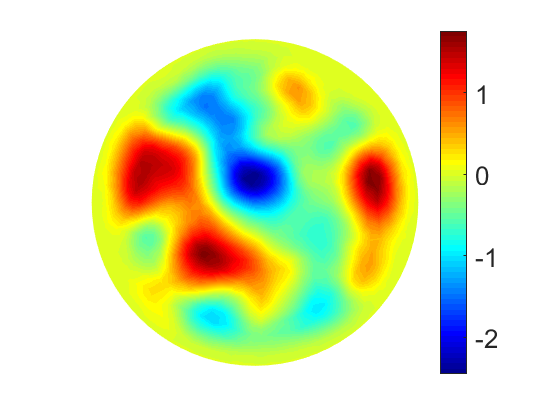}
    \includegraphics[trim=60 20 40 10, clip, width=0.15\textwidth]{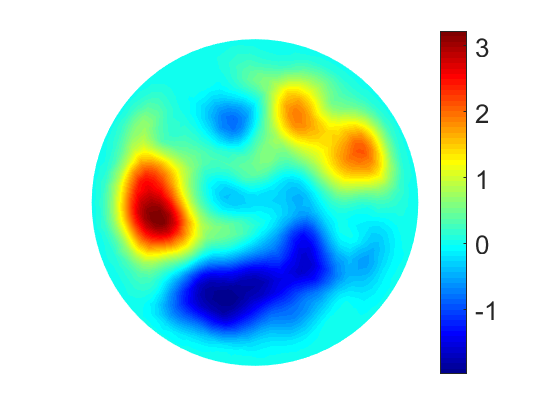}}
    \caption{Left to right: two Mat\'ern prior samples with $\ell = 0.1$, $0.2$ and $0.3$, respectively. In all samples, $\nu=3$.}
    \label{fig:priors}
\end{figure}

We draw $N$ geodesics at random according to the uniform law for $(\alpha,\beta)$ (some samples on $\partial_+ SM$ of size $N= 200, 400, 800$ are visualised Fig.~\ref{fig:densities}), and then generate synthetic data $(Y_i, (X_i,V_i))_{i=1}^N$ as explained in Section \ref{sec:numerics_forward} for the magnetic field $\Phi_0$ displayed in Fig.~\ref{fig:PhiCPhi}, adding Gaussian noise $N(0,\sigma^2)$ to each matrix entry of $C_{\Phi_0}$. 

\begin{figure}[htpb]
    \centering
    \includegraphics[trim=20 20 20 10, clip, width=0.32\textwidth]{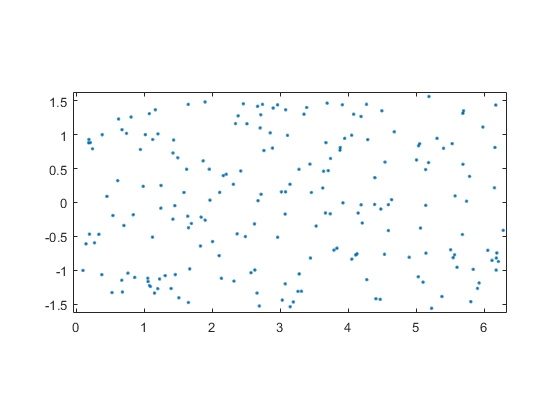} 
    \includegraphics[trim=20 20 20 10, clip, width=0.32\textwidth]{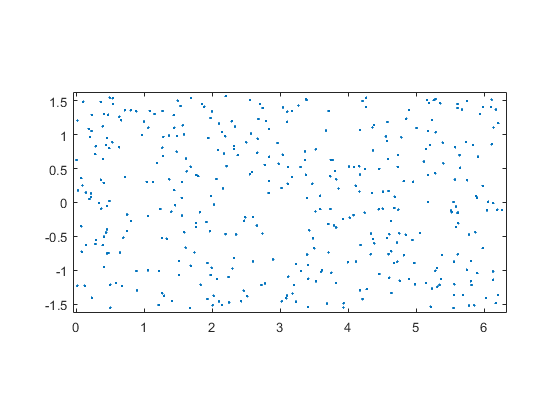} 
    \includegraphics[trim=20 20 20 10, clip, width=0.32\textwidth]{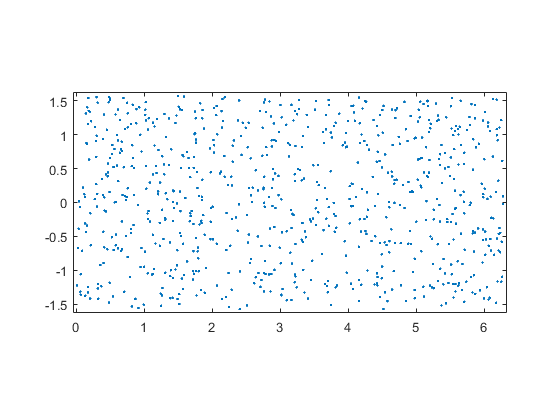}     
    \caption{Left to Right: Examples of sample draws on $\partial_+ SM$ for $N = 200, 400, 800$.}
    \label{fig:densities}
\end{figure}


\smallskip

We then implement the pCN algorithm to approximately compute the posterior mean $\bar \Phi_N =E^\Pi[\Phi|(Y_{i}, (X_i, V_i))_{i=1}^N]$ from Theorem \ref{main}. The stepsize $\delta$ is adjusted so that after `burn-in', the acceptance rate of proposals stabilises around $25\%$. Once the chain is computed we visualise $\widehat{\Phi} = \frac{1}{N_{s}} \sum_{n=0}^{N_s}\Phi_n$ -- examples of outcomes corresponding to increasing data set are given in Fig.~\ref{fig:estimators}, illustrating the improvement in `reconstructions' as the number $N$ of measurement points increases.

\begin{figure}[htpb]
    \centering
    \includegraphics[width = \textwidth]{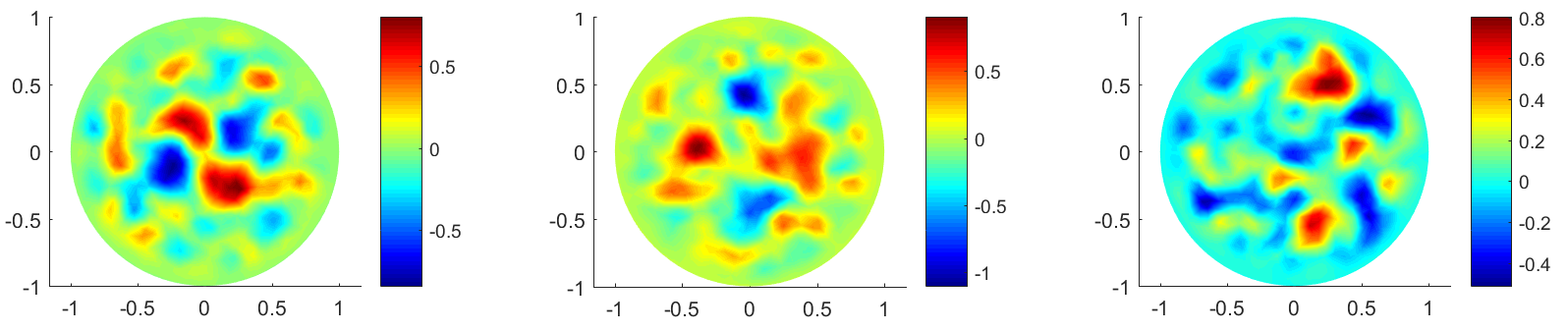} \\
    \includegraphics[width = \textwidth]{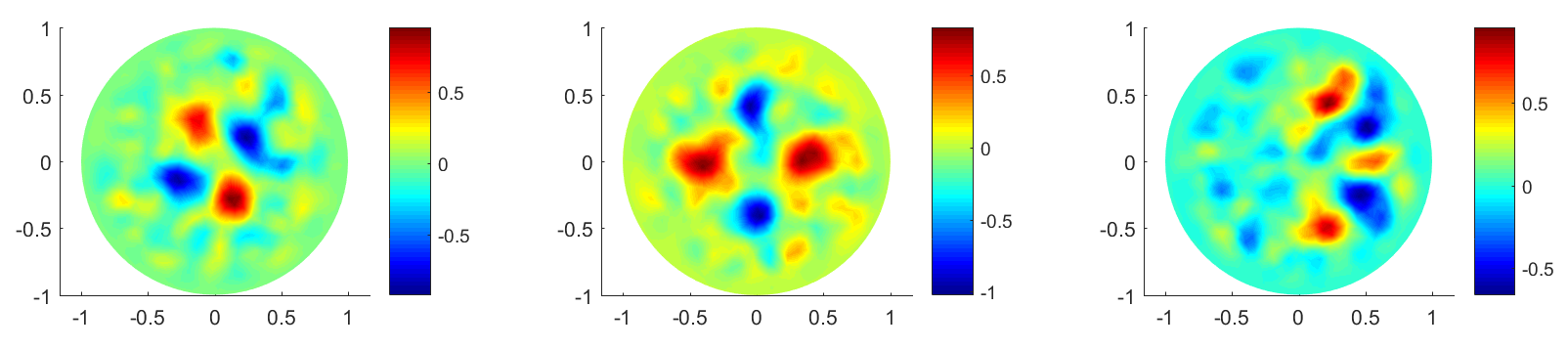} \\
    \includegraphics[width = \textwidth]{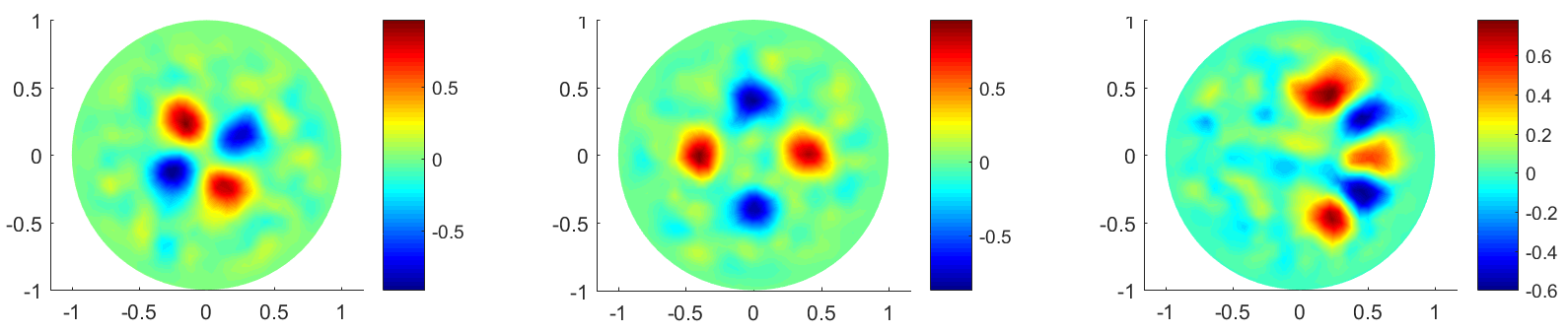}
    \caption{Top to bottom:  The posterior mean field $\widehat{\Phi}$ for sample sizes $N = 200, 400, 800$, to be compared with the true field $\Phi_0$ from Fig.~\ref{fig:PhiCPhi}. The number of Monte-Carlo iterations equals $N_s=100000$. Other parameters: $\sigma = 0.05$, $\nu = 3$, $\ell = 0.2$, $\delta = 0.000025$, $N_v = 886$.}
    \label{fig:estimators}
\end{figure}

\section{Proofs}

\subsection{Geometric preliminaries} \label{eq:geoprelims}

Let $(M,g)$ be a compact oriented two dimensional Riemannian manifold with smooth boundary $\partial M$. As before $SM$ will denote the unit circle bundle which is a compact 3-manifold with boundary given by 
\[ \partial(SM)=\{(x,v)\in SM:\;x\in \partial M\}. \]  
We let $X$ be the geodesic vector field, i.e. the infinitesimal generator of the geodesic flow of $M$. Since $M$ is assumed oriented there is a circle action on the fibers of $SM$ with infinitesimal generator $V$ called the {\it vertical vector field}. It is possible to complete the pair $X,V$ to a global frame
of $T(SM)$ by considering the vector field $X_{\perp}:=[X,V]$. There are two additional structure equations given by $X=[V,X_{\perp}]$ and $[X,X_{\perp}]=-\kappa V$ where $\kappa$ is the Gaussian curvature of the surface. Using this frame we can define a Riemannian metric on $SM$ by declaring $\{X,X_{\perp},V\}$ to be an orthonormal basis and the volume form of this metric will be denoted by $d\Sigma^3$. The fact that $\{ X, X_{\perp}, V \}$ are orthonormal together with the commutator formulas implies that the Lie derivative of $d\Sigma^3$ along the three vector fields vanishes.

Given functions $u,v:SM\to \C^n$ we consider the inner product
\begin{align}
    (u,v) =\int_{SM}\langle u,v\rangle_{\mathbb \C^n}\,d\Sigma^3.    
    \label{eq:L2inner}
\end{align}
Upon defining $\mu(x,v):= -g_x(v,\nu_x)$ for $(x,v)\in \partial SM$, the following formula (known as Santal\'o's formula) holds for any $f\in L^1(SM)$: 
\begin{align}
    \int_{SM} f(x,v)\ d\Sigma^3 = \int_{\partial_+ SM} \int_0^{\tau(x,v)} f(\varphi_t(x,v))\ dt\ \mu(x,v) \ d\Sigma^2,
    \label{eq:santalo}
\end{align}
where $\varphi_{t}$ is the geodesic flow.

We now discuss the manifold $\partial_+ SM$ and its geometry. One may define a natural frame on $\partial_+ SM$, given by 
\begin{align}
    V := V|_{\partial_+ SM}, \qquad T := (\mu_\perp X + \mu X_\perp)|_{\partial_+ SM}, \quad \text{where }\quad \mu_\perp := V\mu
    \label{eq:TV}
\end{align}
($T$ represents horizontal differentiation along the tangent direction). It is easily seen that $[V,T] = 0$ and that these two vector fields are orthonormal for the metric on $\partial SM$ induced by the metric defined on $SM$. In particular $(T,V)$ is an orthonormal frame for $\partial_+ SM$ and we may define $H^s(\partial_+ SM; \cdot)$ with respect to that frame. We now prove a useful lemma that will simplify later calculations.

\begin{Lemma}\label{lem:PVPT} Let $(M,g)$ be a non-trapping surface with strictly convex boundary. Then the vector field $X$ can be completed into a global, pairwise commuting frame $\{X,P_T,P_V\}$ of $T(SM)$. This frame is smooth on $SM\backslash S\partial M$, continuous on $SM$ and satisfies $P_T|_{\partial_+ SM} = T$ and $P_V|_{\partial_+ SM} = V$. 
\end{Lemma}

\begin{proof}[Proof of Lemma \ref{lem:PVPT}]
    For $(x,v)\in \partial_+ S M\backslash S\partial M$ and $t\in (0,\tau(x,v))$, we define two vector fields on $S M^{int}$
    \begin{align*}
	(P_T)_{\varphi_t(x,v)} := d\varphi_t|_{(x,v)} T_{(x,v)}, \qquad (P_V)_{\varphi_t(x,v)} := d\varphi_t|_{(x,v)} V_{(x,v)}.
    \end{align*}
    Since the map $(x,v,t)\mapsto \varphi_t(x,v)$ is smooth and injective for $(x,v)\in \partial_+ S M\backslash S\partial M$ and $t\in (0,\tau(x,v))$, this defines global, smooth sections of $T(SM^{int})$, and so that $X,P_T,P_V$ pairwise commute. Via direct computation of the differential of the flow (see e.g. \cite[Sec. 4.2]{MP}), one may obtain the following expressions on $SM^{int}$
    \begin{align*}
	P_T = (\mu_\perp)_\psi X + \mu_\psi \left( \ba X_\perp - (X\ba) V \right), \qquad P_V =  - \bb X_\perp + (X\bb)  V,
    \end{align*}
    where $\ba, \bb\colon SM \to \mathbb{R}$ satisfy
    \begin{align*}
	X^2 \ba + \kappa \ba = X^2 \bb + \kappa \bb = 0\qquad (SM), \qquad  \left[ \begin{smallmatrix} \ba & \bb \\ X\ba & X\bb \end{smallmatrix}\right]|_{\partial_+ SM} = \id,
    \end{align*}
    and where for $h:\partial_+ SM\to \Cm$, one defines $h_\psi:SM\to \Cm$ though the relation
    \begin{align*}
	h_\psi (\varphi_t(x,v)) = h(x,v), \qquad (x,v)\in \partial_+ SM,\quad t\in [0,\tau(x,v)]. 
    \end{align*}
    One further notices that the definition of $P_V, P_T$ extends by continuity to $\partial(SM)$, with the appropriate restrictions claimed in the statement of the lemma. 
\end{proof}

\subsection{Forward estimates - proof of Theorem \ref{thm:forward}}

In this section, we derive various continuity estimates for the forward map $\Phi\mapsto C_\Phi$. Recall that if the boundary $\partial M$ is strictly convex, by \cite[Lemma 4.1.2 p113]{Sharafutdinov} there is a constant $C_0(M,g)>0$ such that 
\begin{align}
    \tau(x,v) \le C_0\ \mu(x,v), \qquad \forall (x,v)\in \partial_+ SM.
    \label{eq:C0const}
\end{align}
We start with the following basic estimates. 

\begin{Lemma}[Work-horse lemma]\label{lem:workhorse0} Let $(M,g)$ be a non-trapping surface with strictly convex boundary and $\Phi\in C(M,\mathfrak{u}(n))$. Suppose $F\in C(SM,\C^{n\times n})$ and consider the unique continuous solution $G:SM\to\C^{n\times n}$ to $XG + \Phi G = F$ on $SM$ with $G|_{\partial_{-}SM}=0$. Then there exists a constant $C_1(M,g)$ such that 
    \begin{align}
	\norm{G|_{\partial_+ SM}}_{L^{\infty}(\partial_+ SM,\Cm^{n\times n})}&\le \norm{G}_{L^{\infty}(SM,\Cm^{n\times n}) }\le C_1 \norm{F}_{L^{\infty}(SM,\Cm^{n\times n})}, 	\label{eq:basicLinf} \\
	\norm{G}_{L^{2}(SM,\Cm^{n\times n}) }&\le C_1 \norm{F}_{L^{2}(SM,\Cm^{n\times n})}, \label{eq:basicL2_1} \\
	\norm{G|_{\partial_+ SM}}_{L^{2}(\partial_+ SM,\Cm^{n\times n}) } &\le C_1 \norm{F}_{L^{2}(SM,\Cm^{n\times n})}, \label{eq:basicL2_2}
    \end{align}
    The constant $C_1$ can be chosen as $C_1 = \max (\tau_\infty, \sqrt{C_0})$, with $\tau_\infty$ the diameter of $M$ and $C_0$ the constant given in \eqref{eq:C0const}. 
\end{Lemma} 

\begin{proof} It is easy to check that
    \[G(x,v)=-U_{\Phi}(x,v)\int_{0}^{\tau(x,v)}(U^{-1}_{\Phi}F)(\varphi_{t}(x,v))\,dt, \qquad (x,v)\in SM, \]
    where $U_\Phi$ is the unique solution $U$ to $XU + \Phi U = 0$ on $SM$ with $U|_{\partial_+ SM} = \id$. Taking Frobenius norm, using $U(n)$-invariance and the fact that that $U_{\Phi}$ is unitary, we get
    \begin{align*}
	\frob{G}(x,v) = \frob{\int_{0}^{\tau(x,v)}(U^{-1}_{\Phi}F)(\varphi_{t}(x,v))\,dt}&\le \int_0^{\tau(x,v)} \frob{U^{-1}_{\Phi}F} (\varphi_t(x,v))\ dt \\
	&\le \int_0^{\tau(x,v)} \frob{F} (\varphi_t(x,v))\ dt.
    \end{align*}
    Upon bounding the right-hand side crudely by $\tau_\infty \|F\|_{L^\infty}$, this immediately implies \eqref{eq:basicLinf}. On to the $L^2$ estimates, applying Cauchy-Schwarz yields for all $(x,v)\in SM$
    \begin{align}
	\frob{G}(x,v)^2 \le \tau(x,v) \int_{0}^{\tau(x,v)} \frob{F}^2(\varphi_t(x,v))\ dt \le \tau_\infty \int_{\gamma_{x,v}} \frob{F}^2,
	\label{eq:tmpGF}
    \end{align}
    where $\gamma_{x,v}$ is the maximal geodesic passing through $(x,v)$. Now fix $(x,v)\in \partial_+ SM$ and integrate the inequality above along the geodesic flow $\varphi_t(x,v)$ to arrive at 
    \begin{align*}
	\int_{0}^{\tau(x,v)} \frob{G}(\varphi_t(x,v))^2\ dt \le \tau_\infty^2 \int_0^{\tau(x,v)} \frob{F}(\varphi_t(x,v))^2\ dt, \qquad (x,v) \in \partial_+ SM.
    \end{align*}
    Multiplying both sides by $\mu$, integrating w.r.t. $d\Sigma^2$ and using Santal\'o's formula yields \eqref{eq:basicL2_1}. 

    For the estimate on $L^2(\partial_+ SM)$, looking at \eqref{eq:tmpGF} for $(x,v)\in \partial_+ SM$ and using \eqref{eq:C0const}, we arrive at
    \begin{align*}
	\frob{G}(x,v)^2 \le C_0 \int_{0}^{\tau(x,v)} \frob{F}^2(\varphi_t(x,v))\ dt \mu(x,v), \qquad (x,v)\in \partial_+ SM. 
    \end{align*}
    Integrating w.r.t. $d\Sigma^2$ and using Santal\'o's formula \eqref{eq:santalo} on the right hand side immediately gives \eqref{eq:basicL2_2}. Lemma \ref{lem:workhorse0} is proved.
\end{proof}

We now prove the main result on forward estimates, Theorem \ref{thm:forward}. We shall follow the model proof of \cite[Theorem 4.2.1]{Sharafutdinov} which shows that the standard X-ray transform $I$ maps $H^s$ to $H^s$.
We do this in two stages: first we explain in Sec. \ref{proof:supportinside} the proof in the simpler case in which the matrix fields have support contained in the interior of $M$ and then we explain in Sec. \ref{proof:supporteverywhere} how to derive the general case.

\subsubsection{Proof of Theorem \ref{thm:forward} assuming $\Phi$ and $\Psi$ with support in the interior of $M$}\label{proof:supportinside} As a preliminary identity, given $\Phi$ and $\Psi$ two skew hermitian matrix fields, consider the two $U(n)$-valued solutions $U_{\Phi}, U_{\Psi}$ such that $XU_{\Phi} + \Phi U_{\Phi} = 0$ with boundary condition $U_\Phi|_{\partial_- SM} = \id$. It is immediate to find that the relation
    \begin{align*}
	X(U_\Phi-U_\Psi) + \Phi(U_\Phi-U_\Psi) = - (\Phi-\Psi)U_\Psi	
    \end{align*}
    holds pointwise on $SM$, and that $(U_\Phi-U_\Psi)|_{\partial_- SM} = 0$. Using that $(U_\Phi-U_\Psi)|_{\partial_+ SM} = C_\Phi-C_\Psi$ with estimate \eqref{eq:basicLinf} yields 
    \begin{align*}
	\norm{C_\Phi-C_\Psi}_{L^\infty} \le C_1 \norm{(\Phi-\Psi)U_\Psi}_{L^\infty} = C_1 \norm{\Phi-\Psi}_{L^\infty}. 
    \end{align*}
    Similarly, combining the observation with \eqref{eq:basicL2_2} yields \eqref{eq:cont0}, and we can also obtain, using \eqref{eq:basicL2_1}, 
    \begin{align}
	\norm{U_\Phi-U_\Psi}_{L^2} \le C_1 \norm{\Phi-\Psi}_{L^2}. 
	\label{eq:coco1}
    \end{align}

    To prove the $C^1$ continuity estimate, consider the function $W:= P_V(U_\Phi-U_\Psi)$, such that $W|_{\partial_+ SM} = V(C_\Phi-C_\Psi)$ and for brevity set $P=P_{V}$. The following identity is immediate: 
    \begin{align*}
	X W  + \Phi W = - (P\Phi (U_\Phi - U_\Psi) + (P(\Phi-\Psi)) U_\Psi + (\Phi-\Psi)PU_\Psi).
    \end{align*}
    In addition, since $\Phi$ and $\Psi$ are compactly supported in $M^{\text{int}}$, the functions $U_\Phi$, $U_\Psi$ equal the identity matrix in a neighbourhood of $\partial_- SM$ and in particular, $W|_{\partial_- SM} = 0$.

    Using estimates \eqref{eq:basicLinf}-\eqref{eq:basicL2_1}-\eqref{eq:basicL2_2} and $U(n)$-invariance of Frobenius norms gives: 
    \begin{align*}
	\norm{V(C_\Phi-C_\Psi)}_{L^2} &\le C_1 \norm{P\Phi (U_\Phi - U_\Psi) + (P(\Phi-\Psi)) U_\Psi + (\Phi-\Psi)PU_\Psi}_{L^2} \\
	&\le C_1 ( \norm{P\Phi}_\infty \norm{U_\Phi - U_\Psi}_{L^2} + \norm{P(\Phi-\Psi)}_{L^2} + \norm{\Phi-\Psi}_{L^2} \norm{PU_\Psi}_{L^\infty})
    \end{align*}
    We also have $X(PU_\Psi) +\Psi PU_\Psi = - (P\Psi) U_\Psi$ with $PU_{\Psi}|_{\partial_- SM} = 0$, so by \eqref{eq:basicLinf}, we get $\norm{PU_\Psi}_{L^\infty} \le C_1 \norm{P\Psi}_{L^\infty}$. Combining this fact with \eqref{eq:coco1}, we arrive at  
    \begin{align*}
	\norm{V(C_\Phi-C_\Psi)}_{L^2} \le C_1 (C_1 (\norm{P\Phi}_{L^\infty}+\norm{P\Psi}_{L^\infty}) \norm{\Phi-\Psi}_{L^2} + \norm{P(\Phi-\Psi)}_{L^2}),
    \end{align*}
    and a similar bound for $\norm{P_V (U_\Phi-U_\Psi)}_{L^2}$. Obtaining a similar estimate for $T(C_\Phi-C_\Psi)$, we arrive at 
    \begin{align*}
	\norm{C_\Phi-C_\Psi}_{H^1} \lesssim (1+ \norm{\Phi}_{C^1} + \norm{\Psi}_{C^1}) \norm{\Phi-\Psi}_{H^1}. 
    \end{align*}
    Similar arguments using sup norms everywhere yield 
    \begin{align*}
	\norm{C_\Phi-C_\Psi}_{C^1} \lesssim (1+ \norm{\Phi}_{C^1} + \norm{\Psi}_{C^1}) \norm{\Phi-\Psi}_{C^1}. 
    \end{align*}
    To proceed to higher-order derivatives, if $\bP^{\balpha} = P_V^{\alpha_1} P_T^{\alpha_2}$ is a derivative of order $|\alpha|$, setting $W = \bP^{\balpha}(U_\Phi-U_\Psi)$, we have $W|_{\partial_+ SM} = V^{\alpha_1}T^{\alpha_2} (C_\Phi-C_\Psi)$, $W|_{\partial_{-}SM} = 0$ and 
    \begin{align*}
	XW + \Phi W = - [\bP^\balpha, \Phi] (U_\Phi-U_\Psi),
    \end{align*}
    where the right-hand-side involves derivatives of $\Phi$ of order at most $|\balpha|$, and derivatives of $U_\Phi-U_\Psi$ of order at most $|\balpha|-1$. Combining the estimates of Lemma \ref{lem:workhorse0} and an induction on $k$ (whose formulation also involves control on $\norm{P_V^{\alpha_1}P_T^{\alpha_2}(U_\Phi-U_\Psi)}_{L^2(SM)}$ for all $\alpha_1+\alpha_2\le k$, and where the commuting frame $\{X,P_V,P_T\}$ avoids the proliferation of terms due to non-trivial commutators) proves the theorem for higher-order derivatives. 
\qed

\subsubsection{Proof of Theorem \ref{thm:forward} $\Phi$ and $\Psi$ supported up to $\partial M$} \label{proof:supporteverywhere}
 Consider a compact non-trapping surface $(M,g)$ with strictly convex boundary and let $\Phi\in C(M,\C^{n\times n})$ be a matrix-valued field. We shall call $R_{\Phi}\in C(SM,GL(n,\C))$ an {\it integrating factor} for $\Phi$ if
$R_{\Phi}$ is differentiable along the geodesic vector field $X$ and $XR_{\Phi}+\Phi R_{\Phi}=0$. Let $U_{\Phi}$ denote the unique integrating factor with $U_{\Phi}|_{\partial_{-}SM}=\id$. Recall that $C_{\Phi}=U_{\Phi}|_{\partial_{+}SM}$. First note that the work of the previous section also proves for every $k\ge 0$ that if $\Phi$ and $\Psi$ are $C^k$ matrix fields compactly supported inside of $M^{int}$, we also have
\begin{align}
    \begin{split}
	\|U_\Phi-U_\Psi\|_{C^k(SM)} &\lesssim (1+\|\Phi\|_{C^k} + \|\Psi\|_{C^k})^k \|\Phi-\Psi\|_{C^k(M)},  \\
	\|U_\Phi-U_\Psi\|_{H^k(SM)} &\lesssim (1+\|\Phi\|_{C^k} + \|\Psi\|_{C^k})^k \|\Phi-\Psi\|_{H^k(M)}.	
    \end{split}
    \label{eq:estUphi}
\end{align}

Let $\alpha:\partial SM\to\partial SM$ denote the scattering relation of the metric (i.e. the map that takes initial conditions of a geodesic at the moment of entry to final conditions at the moment of exit). If $R_\Phi$ denotes any other integrating factor for $\Phi$, then it must have the form $U_\Phi F^{\sharp}$, where $F^{\sharp}$ is the first integral  (i.e. $XF^{\sharp}=0$) determined by $F\in C(\partial_{+}SM,GL(n,\C))$. Thus $R_{\Phi}=U_{\Phi}F^{\sharp}$ and from this we deduce
\begin{equation}
    C_{\Phi}=R_{\Phi}(R_{\Phi}^{-1}\circ\alpha).
    \label{eq:CintermsofR}
\end{equation}
In particular, given two continuous matrix fields $\Phi$, $\Psi$, Equation \eqref{eq:CintermsofR} implies the identity on $\partial_+ SM$: 
\begin{align}
    C_\Phi - C_\Psi = (R_{\Phi}-R_\Psi) R_{\Phi}^{-1}\circ\alpha + R_{\Psi} (R_{\Phi}^{-1}-R_{\Psi}^{-1})\circ \alpha.
    \label{eq:cruxExtension}
\end{align}
To complete the proof of Theorem \ref{thm:forward} for $\Phi,\Psi$ supported up to the boundary, we then need to construct integrating factors with good regularity on $SM$ (i.e, at $\partial_0 SM$ included) and which behave continuously in terms of $\Phi$ and $\Psi$. To this end, we consider $(M,g)$ isometrically embedded in a closed manifold $(S,g)$. The Seeley extension theorem asserts that for any $k\geq 0$ there is a continuous extension map
\[E_{k}:C^{k}(M)\to C^{k}(S), \qquad E_{k} \colon H^k(M) \to H^k(S), \]
(It also works for $C^{\infty}$.)
We consider a slightly larger compact manifold with boundary $\widetilde{M}\subset S$ engulfing $M$ so that $(\widetilde{M},g)$ stays non-trapping and with strictly convex boundary.
We fix once and for all a smooth cut off function $\chi$ so that it has compact support in $\widetilde{M}^{int}$
and it equals $1$ near $M$. Thus given $\Phi\in C^{k}(M,\mathfrak{u}(n))$,
\[\widetilde{\Phi}:=\chi E_{k}(\Phi)\in C_{c}^{k}(\widetilde{M},\mathfrak{u}(n)),\]
and since $E_{k}$ is continuous,
\begin{equation}
    \norm{\widetilde{\Phi}}_{C^{k}}\lesssim \norm{\Phi}_{C^{k}}, \qquad \norm{\widetilde{\Phi}}_{H^{k}}\lesssim \norm{\Phi}_{H^{k}}.
    \label{eq:extension}
\end{equation}
Now by virtue of the work in Subsection \ref{proof:supportinside} applied to $\widetilde{\Phi}$ on $\widetilde{M}$, we can deduce estimates of the form  
\begin{align}
    \norm{U_{\widetilde{\Phi}}-\text{\rm id}}_{C^{k}}\lesssim\norm{\widetilde{\Phi}}^{k}_{C^{k}}, \qquad \norm{U^{-1}_{\widetilde{\Phi}}-\text{\rm id}}_{C^{k}}\lesssim\norm{\widetilde{\Phi}}^{k}_{C^{k}}.    
    \label{eq:boundU}
\end{align}
We then take as smooth integrating factors $R_{\Phi}:=U_{\widetilde{\Phi}}|_{SM}$ and $R_{\Psi}:=U_{\widetilde{\Psi}}|_{SM}$. Combining \eqref{eq:extension} and \ref{eq:boundU} we derive
\begin{equation}
    \norm{R_{\Phi}-\id}_{C^{k}}\lesssim\norm{\Phi}^{k}_{C^{k}}, \qquad \norm{R^{-1}_{\Phi}-\id}_{C^{k}}\lesssim\norm{\Phi}^{k}_{C^{k}}.
    \label{eq:R}
\end{equation}
Combining \eqref{eq:extension} and \eqref{eq:estUphi} applied to $U_{\widetilde{\Phi}}$ and $U_{\widetilde{\Psi}}$, we obtain 
\begin{align}
    \begin{split}
	\|R_\Phi - R_{\Psi}\|_{C^k} \le \|U_{\widetilde{\Phi}} - U_{\widetilde{\Psi}}\|_{C^k} &\lesssim (1+\|\widetilde{\Phi}\|_{C^k} + \|\widetilde{\Psi}\|_{C^k})^k \|\widetilde{\Phi}-\widetilde{\Psi}\|_{C^k(M)} \\
	&\lesssim (1+\|\Phi\|_{C^k} + \|\Psi\|_{C^k})^k \|\Phi-\Psi\|_{C^k(M)},	
    \end{split}
    \label{eq:Rdiff}    
\end{align}
and similarly for $\|R_\Phi^{-1} - R_{\Psi}^{-1}\|_{C^k}$, and for $H^k$ norms. Then the proof for Theorem \ref{thm:forward} for $\Phi,\Psi$ supported up to the boundary consists in applying the product rule to \eqref{eq:cruxExtension} and using estimates \eqref{eq:R} and \eqref{eq:Rdiff}.

\subsection{Stability estimate - proof of Theorem \ref{thm:stability}}

\subsubsection{Setting, main results and proofs of Theorem \ref{thm:stability} and Corollary \ref{cor:stab2}} Before considering the non-linear inverse problem, we must establish a stability estimate for a linear inverse problem, that of reconstructing a function $f\in C^\infty(M,\Cm^n)$ from its {\bf attenuated X-ray transform}, where the attenuation is matrix-valued. Namely, given $\Phi$ a smooth skew-hermitian matrix in $M$, we define $I_\Phi f := u^f|_{\partial_+ SM}$, where $u=u^{f}:SM\to \Cm^n$ is the unique solution to the problem
\begin{align*}
    Xu + \Phi u = -f \qquad (SM), \qquad u|_{\partial_- SM} = 0.
\end{align*}

The injectivity of such a transform was proved in \cite{PSUGAFA}, and we now provide a stability estimate for it. 

\begin{Theorem} \label{thm:stabLinear} Let $(M,g)$ be a simple Riemannian surface with boundary and $\Phi$ a smooth, skew-hermitian matrix field in $M$. Then for any $f\in C^\infty(M)$, we have the following stability estimate
    \begin{align}
	\|f\|_{L^2(M,\Cm^n)} \le C_1 (1+\|\Phi\|_{C^1}) e^{C_2 \|\Phi\|_{C^1}} \|I_\Phi f\|_{H^1(\partial_+ SM, \Cm^n)}.
	\label{eq:stabLinear}
    \end{align}        
\end{Theorem}



\begin{Remark}[Dependence of $C_1,C_2$] {\rm The constants $C_1,C_2$ only depend on the geometry of $(M,g)$. The constant $C_1$ blows up like $(\beta-1)^{-1}$, where $\beta$ is the terminator constant of $(M,g)$. This is one of the ways that this stability estimate ceases to hold as one approaches non-simplicity. The main other quantity appearing in $C_1, C_2$ is $w_\infty$, the sup norm of the integrating factor defined below. The behavior of such a quantity, while finite on any simple surface, remains to be better understood.   } 
\end{Remark}

On to the non-linear stability estimate, injectivity of the operator $\Phi\to C_\Phi$ restricted to $\u(n)$-valued fields was initially proved in \cite{PSUGAFA}, and Theorem \ref{thm:stability} upgrades this result with a stability estimate. While the remaining sections will focus on the proof of Theorem \ref{thm:stabLinear}, we now explain how this result implies Theorem \ref{thm:stability}. The main additional ingredient needed is a pseudo-linearization identity, relating scattering data to attenuated X-ray transforms: 

\begin{Lemma}[Pseudo-linearization]\label{lem:pseudolin}
    Let $(M,g)$ be a non-trapping surface with strictly convex boundary. For any $\Phi, \Psi\in C(M,\Cm^{n\times n})$, the following relation holds
    \begin{equation}
	C_{\Phi}C_{\Psi}^{-1}=\text{\rm \id}+I_{\Theta(\Phi,\Psi)}(\Phi-\Psi),
	\label{eq:two}
    \end{equation}
    where $I_{\Theta(\Phi,\Psi)}\colon L^2(M, \Cm^{n\times n})\to L^2(\partial_+ SM, \Cm^{n\times n})$ is an attenuated X-ray transform with matrix field $\Theta(\Phi, \Psi)$, an endomorphism of $\Cm^{n\times n}$ with pointwise action 
    \begin{align*}
	\Theta(\Phi, \Psi)\cdot U = \Phi U - U\Psi, \qquad U\in \Cm^{n\times n}.
    \end{align*}
\end{Lemma}

\begin{proof}[Proof of Lemma \ref{lem:pseudolin}] With $U_\Phi, U_\Psi$ the fundamental solutions of $XU_\Phi + \Phi U_\Phi = 0$ with $U_\Phi|_{\partial_- SM} = \id$ and $U_\Phi|_{\partial_+ SM} = C_\Phi$ (similarly for $\Psi$), denote $W:= U_\Phi U_\Psi^{-1} - \id$. A direct computation shows that
    \begin{align*}
	XW + \Phi W - W\Psi = - (\Phi-\Psi) \qquad (SM), \qquad W|_{\partial_- SM} = 0,
    \end{align*}
    and thus by the definition of the attenuated X-ray transform, $W|_{\partial_+ SM} = I_{\Theta(\Phi,\Psi)}(\Phi-\Psi)$. Since we also have by construction $W|_{\partial_+ SM} = C_\Phi C_{\Psi}^{-1}-\id$, identity \eqref{eq:two} follows.     
\end{proof}

\begin{proof}[Proof of Theorem \ref{thm:stability}] Appealing to the pseudo-linearization \eqref{eq:two}, one may notice that if $\Phi$, $\Psi$ are skew-hermitian, then the field $\Theta(\Phi,\Psi)$ is skew-hermitian when viewed as an endomorphism of $\Cm^{n\times n}$. Moreover, since the entries of $\Theta(\Phi,\Psi)$ are linear in the entries of $\Phi$ and $\Psi$, we directly have that
    \begin{align*}
	\|\Theta(\Phi,\Psi)\|_{C^1} \le C (\|\Phi\|_{C^1} \vee \|\Psi\|_{C^1}), 
    \end{align*}
    with $C$ a universal constant. Then relation \eqref{eq:two}, together with Theorem \ref{thm:stabLinear} immediately implies
    \begin{align*}
	\|\Phi-\Psi\|_{L^2(M,\Cm^{n\times n})} &\le C_1 (1+ \|\Theta(\Phi,\Psi)\|_{C^1})\ e^{C_2 \|\Theta(\Phi,\Psi)\|_{C^1}} \|I_{\Theta(\Phi,\Psi)}(\Phi-\Psi)\|_{H^1(\partial_+ SM)} \\
	&\le C_1' (1+ \|\Phi\|_{C^1}\vee \|\Psi\|_{C^1})\ e^{C_2' (\|\Phi\|_{C^1} \vee \|\Psi\|_{C^1})} \|C_{\Phi}C_{\Psi}^{-1} - \id\|_{H^1(\partial_+ SM)}.
    \end{align*}    
    This shows Theorem \ref{thm:stability} when $\Phi,\Psi\in C^\infty(M, {\mathfrak u}(n))$. Since all quantities involved above do not depend on derivatives of $\Phi, \Psi$ of order higher than $1$, and $C_1, C_2$ are independent of $\Phi,\Psi$, approximating $\Phi,\Psi\in C^1(M, {\mathfrak u}(n))$ by sequences in $C^\infty(M, {\mathfrak u}(n))$ (and using Theorem \ref{thm:forward})  will yield the same stability estimate for $C^1$ matrix fields.
\end{proof}

We also cover the proof of Corollary \ref{cor:stab2}, based on the previous result and the forward estimate Theorem \ref{thm:forward}. 

\begin{proof}[Proof of Corollary \ref{cor:stab2}] It is enough to show that 
    \begin{align}
	\|C_\Phi C_\Psi^{-1} - \id\|_{H^1} \lesssim (1 + \|\Psi\|_{C^1}) \norm{C_{\Phi}-C_{\Psi}}_{H^{1}}.
	\label{eq:wts}
    \end{align}
    To show this, we write at the pointwise level: 
    \begin{align*}
	\frob{C_\Phi C_{\Psi}^{-1} - \id} = \frob{(C_\Phi - C_\Psi)C_{\Psi}^{-1}} = \frob{C_\Phi - C_\Psi},
    \end{align*}
    hence $\|C_\Phi - C_\Psi\|_{L^2} = \|C_\Phi C_\Psi^{-1} - \id\|_{L^2}$. To control first derivatives, take $P = V$ or $T$, we have 
    \begin{align*}
	\frob{P(C_\Phi C_\Psi^{-1} - \id)} &= \frob{P (C_\Phi - C_\Psi) + (\id-C_\Phi C_\Psi^{-1})PC_\Psi} \\
	&\le \frob{P(C_\Phi - C_\Psi)} + \frob{PC_\psi} \frob{\id - C_\Phi C_\Psi^{-1}} 
    \end{align*}
    using triangle inequality and submultiplicativity. Squaring, taking the sup norm of $\frob{PC_\Psi}$ and integrating on $\partial_+ SM$, we obtain
    \begin{align*}
	\|P(C_\Phi C_\Psi^{-1} - \id)\|_{L^2}^2 \le 2 (\|P(C_\Phi - C_\Psi)\|_{L^2}^2 + \|PC_\Psi\|_\infty^2 \|C_\Phi-C_\Psi\|_{L^2}^2). 
    \end{align*}
    Combining the estimates for $P=V$ and $P = T$ we arrive at
    \begin{align*}
	\|C_\Phi C_\Psi^{-1} - \id\|_{H^1}^2 &\le (1 + 2\|VC_\Psi\|_{L^\infty}^2 + 2\|TC_\Psi\|_{L^\infty}^2) \|C_\Phi-C_\Psi\|_{L^2}^2 \\
	& \qquad + 2 \|V(C_\Phi-C_\Psi)\|_{L^2}^2 + 2 \|T(C_\Phi-C_\Psi)\|_{L^2}^2.
    \end{align*}
    Now using the forward estimate \eqref{eq:continf} with $k=1$ and $\Phi \equiv 0$ (thus $C_\Phi = \id$), we deduce that 
    \begin{align*}
	1 + 2\|VC_\Psi\|_{L^\infty}^2 + 2\|TC_\Psi\|_{L^\infty}^2 \lesssim 1+ \|\Psi\|_{C^1}^2.
    \end{align*}
    This yields the estimate $\|C_\Phi C_\Psi^{-1} - \id\|_{H^1}^2 \lesssim (1 + \|\Psi\|_{C^1}^2) \norm{C_{\Phi}-C_{\Psi}}^2_{H^{1}}$, and taking squareroots yields \eqref{eq:wts} (using that $\sqrt{1+x^2}/(1+x)$ is uniformly bounded for $x\in [0,\infty)$). 
\end{proof}

\subsubsection{Proof of Theorem \ref{thm:stabLinear} - Main outline} As in \cite{PSUGAFA}, the main method of proof involves an energy identity (or Pestov identity), based on integrations by parts on $SM$. To do this, let us recall that with the inner product $(u,v)$ defined in \eqref{eq:L2inner}, and upon also denoting 
\begin{align*}
    (u,v)_{\partial SM} := \int_{\partial SM} u \overline{v} d\Sigma^2, 
\end{align*}
the following integrations by parts formulas holds for $u, v \in C^{\infty}(SM,\C^n)$: 
\begin{align}
    \begin{split}
	(Vu, v) &= -(u,Vv), \qquad (Vu, v)_{\partial SM} = (u, Vv)_{\partial SM}, \\
	(Xu, v) &= -(u, Xv) + (\mu u, v)_{\partial_+ SM}, \qquad \mu(x,v) := - \langle v, \nu_x\rangle.
    \end{split} 
    \label{eq:IBP}    
\end{align}

We will also use extensively the harmonic decomposition on the fibers of $SM$. Namely, the space $L^{2}(SM,\C^n)$ decomposes orthogonally as a direct sum
\[L^{2}(SM,\C^n)=\bigoplus_{k\in\mathbb Z}H_{k}\]
where $H_k$ is the eigenspace of $-iV$ corresponding to the eigenvalue $k$. A function $u\in L^{2}(SM,\C^n)$ has a Fourier series expansion 
\[u=\sum_{k=-\infty}^{\infty}u_{k},\] 
where $u_{k}\in H_k$. Let $\Omega_{k}=C^{\infty}(SM,\C^n)\cap H_{k}$. Of special interest are the operators 
\begin{align}
    \eta_\pm := \frac{1}{2} (X\pm iX_\perp), 
    \label{eq:etapm}
\end{align}
with the property that $\eta_\pm (\Omega_k)\subset \Omega_{k\pm 1}$ for all $k\in \mathbb{Z}$. For more details on the operators $\eta_{\pm}$ and the Fourier expansion we refer to \cite{GK80} where these tools were first introduced.

\begin{Definition} A function $u:SM\to\C^n$ is said to be holomorphic if $u_{k}=0$ for all $k<0$. Similarly, $u$ is said to be antiholomorphic if $u_{k}=0$ for all $k>0$.
\end{Definition}

To control the terms involving the matrix field, one must introduce an artificial connection as we will see below. This first requires that we derive a Pestov identity for X-ray transforms with connection $A$ and matrix\footnote{The matrix field $\Phi$ is also referred to as a 'Higgs' field in the literature.} field $\Phi$. Namely, given a skew hermitian pair $(A,\Phi)$ on the bundle $M\times \Cm^n$ and $f\in C^\infty(M,\Cm^n)$, we define $I_{A,\Phi} f = u|_{\partial_+ SM}$, where $u$ is the unique solution to the problem
\begin{align*}
    Gu = -f \qquad (SM), \qquad u|_{\partial_{-}SM} = 0, \qquad (G:= X+A+\Phi).
\end{align*}

While previous Pestov identities have been derived in \cite{PSUGAFA}, the present one accounts for nonzero boundary terms, and in particular reflects more precisely how the stability constant degrades as $(M,g)$ approaches non-simplicity. This is captured by the concept of {\em terminator constant} $\beta_{\text{Ter}}$: given a simple surface $(M,g)$, there exists a number $\beta_{\text{Ter}}>1$ such that for any $\beta\in (1,\beta_{\text{Ter}}]$, there exists a smooth function $r = r_\beta:SM\to {\mathbb R}$, solution to the Riccati type equation $Xr + r^2 + \beta \kappa = 0$.

\begin{Theorem}\label{thm:pestovid} Let $(M,g)$ a simple surface with boundary, with terminator constant $\beta_{\text{Ter}}>1$, and $(A,\Phi)$ a skew-hermitian pair on the bundle $M\times \C^n$. Then for any $u\in C^\infty(SM, \C^n)$ and $\beta\in (1,\beta_{\text{Ter}}]$, the following identity holds: 
    \begin{align}
	\begin{split}
	    \frac{1}{\beta} \|GVu &- r_\beta Vu\|^2 + \frac{\beta-1}{\beta} \|GVu\|^2 + \|Gu\|^2 - \|VG u\|^2  \\
	    & - (\star F_A u, Vu) - \Re (\Phi u,Gu) - \Re( (\star d_A\Phi)u, Vu) \\
	    &= \Re(\nabla_{T,A}u, Vu)_{\partial SM} + \Re(\langle v^\perp, \nu \rangle \Phi Vu, u)_{\partial SM} - \frac{1}{\beta} (\mu\ r_\beta Vu, Vu)_{\partial SM}.
	\end{split}
	\label{eq:Pestov3}
    \end{align}
\end{Theorem}
In the identity above, 
\begin{align}
    \star d_A \Phi := X_\perp \Phi + [\Phi, V(A)], \qquad \nabla_{T,A}u = Tu + A(x,\nu^\perp) u,    
    \label{eq:Aquantities}
\end{align}
and $r$ is a smooth function on $SM$ which only depends on the surface. The quantity $\star F_A$ is the curvature of the connection $A$, which upon a {\bf judicious choice of connection}, can have a controlled sign. To achieve this, consider the scalar Hermitian connection $a := i\varphi \id$, where $\varphi$ is a smooth 1-form such that $d\varphi = \omega_g$ (the area form of the metric $g$). We choose a specific $\varphi$ of the form $\varphi = \star d h$ for $h$ a real-valued function satisfying $\star d\star d h = 1$ with Neumann condition $dh(\nu) = 0$ at the boundary. The latter condition implies that $\nabla_{T,sa} u = T u$ for any real $s$. Then we have
\begin{align*}
    a = i (X_\perp h)\ \id. \qquad a_1 = \eta_+ h, \qquad a_{-1} = -\eta_- h = -\overline{a_1},
\end{align*} 
with $\eta_\pm$ defined in \eqref{eq:etapm}, and $i\star F_{a} = - 1$. 

By \cite{PSUGAFA}, we can construct a holomorphic scalar function $w\in C^\infty(SM)$ satisfying $Xw = -i X_\perp h$. Without loss of generality, $w$ can be chosen even. The condition on $w_0$ reads $\eta_- (w_0-h) = 0$, for which it is sufficient to use $w_0 = h$. With this choice of $a$ and $s\in \mathbb{R}$, in what follows, we will denote $G_s:= X + sa + \Phi$ and $G = G_0$. With $w$ as above, we have $G_s u = e^{sw} G (e^{-sw}u)$. Moreover, $\overline{w}$ (the complex-conjugate of $w$) is antiholomorphic and solves $X \overline{w} = + i X_\perp h$, so also $G_s u = e^{-s\overline{w}} G (e^{s\overline{w}}u)$.

Lastly, we will denote $\Pi_\pm$ the projection onto positive and negative harmonics. Namely, $\Pi_\pm u = \sum_{\pm k>0} u_k$. We have the following commutators formulas, for any $u\in C^\infty(SM)$: 
\begin{align*}
    [\Pi_-, X+sa+\Phi] u &= (\eta_- + sa_{-1}) u_0 - (\eta_+ + sa_1) u_{-1}, \\
    [\Pi_+, X+sa+\Phi] u &= (\eta_+ + sa_1) u_0 - (\eta_- + sa_{-1}) u_{1}.
\end{align*}

The following lemma will help us controlling $u$ by versions of $u$ which are conjugated by special integrating factors. 

\begin{Lemma}\label{lem:HIF}
    With the holomorphic function $e^{sw}$ and antiholomorphic function $e^{-s'\bar{w}}$ and any $s, s'\in \mR$, we have 
    \begin{align*}
	\Pi_- u = \Pi_- (e^{-sw} \Pi_- (e^{sw}u))), \qquad \Pi_+ u = \Pi_+ (e^{s'\bar{w}}\Pi_+ (e^{-s'\bar{w}}u)),
    \end{align*} 
    in particular we get the equality
    \begin{align}
	u = u_0 + \Pi_- (e^{-sw} \Pi_- (e^{sw}u))) + \Pi_+ (e^{s'\bar{w}}\Pi_+ (e^{-s'\bar{w}}u)).
	\label{eq:HIF}
    \end{align}
\end{Lemma}

\begin{proof} We only prove $\Pi_- u = \Pi_- (e^{-sw} \Pi_- (e^{sw}u)))$, and the rest is similar. It is enough to notice that for any holomorphic function $f$, the equality $\Pi_- (fu) = \Pi_{-} (f \Pi_- u)$ holds, as this amounts to saying that the negative harmonics of $fu$ do not depend on the non-negative harmonics of $u$. This is immediate since
    \begin{align*}
	(fu)_k = \sum_{p\ge 0} f_p u_{k-p}. 
    \end{align*}
    Then we compute immediately
    \begin{align*}
	\Pi_- (e^{-sw} \Pi_- (e^{sw}u))) = \Pi_- \Pi_- (e^{-sw} e^{sw} u) = \Pi_- u, 
    \end{align*}    
    hence the result.
\end{proof}

\textbf{Outline of proof of Theorem \ref{thm:stabLinear}} At first we are going to assume that
the solution $u$ to the transport problem $Xu+\Phi u=-f$, $u|_{\partial_{-}SM}=0$ is $C^{\infty}$.
If $f$ is supported all the way to the boundary, this may not be the case, as $u$ may fail to be smooth at the glancing $\partial_{0}SM$ because $\tau$ is not smooth at $\partial_{0}SM$. However, there is a standard way to fix this issue and we shall do this at the very end.
For now we will proceed as if $u$ were smooth in $SM$.

The initial transport equation, projected onto the harmonic term of degree $0$, reads
\begin{align*}
    -f = \eta_+ u_{-1} + \eta_- u_1 + \Phi u_0 = (\eta_+ u_{-1} + \Phi u_0/2) + (\eta_- u_1 + \Phi u_0/2),
\end{align*}   
so that, in particular, 
\begin{align}
    \|f\|^2 \le 2 \left( \|\eta_+ u_{-1} + \Phi u_0/2\|^2 + \|\eta_- u_1 + \Phi u_0/2\|^2 \right).
    \label{eq:estf}
\end{align}
The crux is then to find how to bound the quantities on the right by the boundary values of $u$. Using a Pestov identity with a special connection $sa$ defined as above (and its holomorphic integrating factor $e^{sw}$), we show how to control the first term using control over $\Pi_- (e^{sw}u))$ for $s>0$. Similar work can be done, to control the second term using control over $\Pi_+ (e^{-s'\bar{w}}u)$ for $s'<0$. 

We first derive in Sec. \ref{sec:miscids} the identity: 
\begin{align}
    \begin{split}
	\eta_+ u_{-1} + \frac{1}{2} \Phi u_0 &= ((e^{sw}u) (\eta_+ - sa_1) (e^{-sw}))_0 + \frac{1}{2} e^{-sw_0} \Phi (e^{sw}u)_0 \\
	&\qquad + \frac{i}{2} (e^{-sw} G_s V\Pi_- (e^{sw}u))_0 + \frac{i}{2} e^{-sw_0} (G_s V\Pi_- (e^{sw}u))_0.
    \end{split}
    \label{eq:temp}
\end{align}
Since $(\eta_+ - sa_1) (e^{-sw})$ only has strictly positive harmonic terms, the first term in the right-hand side of \eqref{eq:temp} only depends on $\Pi_-(e^{sw}u)$. Upon defining $v_s := \Pi_-(e^{sw}u)$, the identity \eqref{eq:temp} reads
\begin{align}
    \begin{split}
	\eta_+ u_{-1} + \frac{1}{2} \Phi u_0 &= (v_s (\eta_+ - sa_1) (e^{-sw}))_0 + \frac{1}{2} \Phi (e^{s(w-w_0)}u)_0 \\    
	&\qquad\qquad + \frac{i}{2} (e^{-sw} G_s V v_s)_0 + \frac{i}{2} e^{-sw_0} (G_s V v_s)_0.	
    \end{split}
    \label{eq:temp2}    
\end{align}
Denoting $w_\infty = \sup_{SM} |w|$, we straightforwardly obtain the estimate
\begin{align}
    \begin{split}
	\|\eta_+ u_{-1} + \frac{1}{2} \Phi u_0\|^2 &\le C_0 \Big( |w|_{C^1}^2 s e^{2sw_\infty} \|v_s\|^2  \\
	&\qquad\qquad + |\Phi|_{C^0}^2 \|(e^{s(w-w_0)}u)_0\|^2 + e^{2sw_\infty} \|G_s Vv_s\|^2 \Big),	
    \end{split}    
    \label{eq:est1}
\end{align}
and control on $\|\eta_+ u_{-1} + \frac{1}{2} \Phi u_0\|^2$ will be obtained after controlling each term in the last right hand side. We first control $\|(e^{s(w-w_0)}u)_0\|^2$ by $\|v_s\|^2 + \|G_sVv_s\|^2$, via the estimate
\begin{align}
    \|(e^{s(w-w_0)}u)_0\|_{L^2(M)} \le C' e^{2sw_\infty} \left(  \|G_s Vv_s\|_{L^2(M)}^2 + |\Phi|_{C^0}^2 \|v_s\|_{L^2(M)}^2 + \|I_\Phi f\|_{L^2(\partial SM)} \right).
    \label{eq:estu0}
\end{align}
We then control $\|v_s\|^2$ and $\|G_sVv_s\|^2$ by boundary terms via Pestov identity and setting up an appropriate threshold on $s$. To do this, we consider the transport problem for $v_s$, written as:
\begin{align*}
    G_s v_s = G_s (\Pi_- (e^{sw}u)) &= [G_s, \Pi_-] (e^{sw}u) \\ 
    &= (\eta_+ + s a_1) (e^{sw}u)_{-1} - (\eta_- + sa_{-1}) ((e^{sw}u)_{0}) 
\end{align*}
We then use the Pestov identity \eqref{eq:Pestov3} for $v_s$, with $\star F_{sa} = is \id$ and $\star d_{sa} \Phi = X_\perp \Phi$: 
\begin{align}
    \begin{split}
	\frac{1}{\beta} \|G_sVv_s &- rVv_s\|^2 + \frac{\beta-1}{\beta} \|G_sVv_s\|^2 + \|(\eta_+ + s a_1) (e^{sw}u)_{-1}\|^2 \\
	&+ s(v_s, iVv_s) - \Re (\Phi v_s,G_s v_s) - \Re( (X_\perp \Phi)v_s, Vv_s) \\
	&= \Re(T v_s, Vv_s)_{\partial SM} + \Re(\langle v^\perp, \nu \rangle \Phi Vv_s, v_s)_{\partial SM} - \frac{1}{\beta} (\mu\ rVv_s, Vv_s)_{\partial SM}. 
    \end{split}
    \label{eq:Pestov_vs1}
\end{align}
Before choosing $s$ appropriately, we need additional work (tedious as in \cite{PSUGAFA}) on the term $\Re (\Phi v_s,G_s v_s)$. Taking into account boundary terms, and upon defining $B_{\pm 1} := \eta_\pm \Phi$, we prove in Sec. \ref{sec:miscids} that
\begin{align}
    \begin{split}
	\Re (\Phi v_s,G_s v_s) &= \sum_{k=1}^\infty (-1)^k \Big( |\Phi (v_s)_{-k}|^2 - \Re (B_{-1} (v_s)_{-k}, (v_s)_{-k-1}) \\
	& \qquad\qquad\qquad + \Re (e_x(v)\Phi (v_s)_{-k}, (v_s)_{-k-1})_{\partial SM}  \Big),
    \end{split}
    \label{eq:tedious}
\end{align}
with $e_x(v)$ defined in \eqref{eq:exv}. The last term in the sum will move to the right-hand side of \eqref{eq:Pestov_vs1}, while the other two need to be controlled with a large $s$. To achieve this, we prove in Sec. \ref{sec:miscids} the following: 
\begin{Lemma}\label{lem:s}
    There exists a universal constant $C>0$ such that for all $s \ge C |\Phi|_{C^1}$, 
    \begin{align}
	\begin{split}
	    s(v_s, iVv_s) &- \sum_{k=1}^\infty (-1)^k \left( |\Phi (v_s)_{-k}|^2 - \Re (B_{-1} (v_s)_{-k}, (v_s)_{-k-1})\right) \\
	    &- \Re( (\star d_{sa}\Phi)v_s, Vv_s) \ge 0.	    
	\end{split}	
	\label{eq:sbound}
    \end{align}	    
\end{Lemma}
In particular, for $s = C |\Phi|_{C^1}+1$, identity \eqref{eq:Pestov_vs1} becomes
\begin{align}
    \begin{split}
	\frac{1}{\beta} \|G_sVv_s&-rVv_s\|^2 + \frac{\beta-1}{\beta} \|G_sVv_s\|^2 + \|(\eta_+ + s a_1) (e^{sw}u)_{-1}\|^2 + \sum_{k=1}^\infty k |v_{-k}|^2 \\
	&\le \Re(T v_s, Vv_s)_{\partial SM} + \Re(\mu_\perp\ \Phi Vv_s, v_s)_{\partial SM} \\
	&\qquad- \frac{1}{\beta} (\mu\ rVv_s, Vv_s)_{\partial SM} - \sum_{k=1}^\infty (-1)^k \Re (e_x(v)\Phi (v_s)_{-k}, (v_s)_{-k-1})_{\partial SM}
    \end{split}
    \label{eq:partial}
\end{align}
We now explain how to bound the right-hand side in terms of $\|I_\Phi f\|^2_{H^1(\partial_+ SM)}$. Recall that $v_s = \Pi_- (e^{sw}u)$. The first claim is that $[\Pi_-,V] = [\Pi_-,T] = 0$. The first one is obvious because both operators are diagonal of the fiberwise Fourier decomposition $C^\infty(\partial SM) = \bigoplus_{k\in \mathbb{Z}} \ker (id - ikV)$. That $T$ is also diagonal on this decomposition follows from the fact that $[T,V] = 0$. With this in mind, we have, on $\partial SM$:  
\begin{align*}
    V v_s = \Pi_- V(e^{sw} u) = \Pi_-  e^{sw} (s (Vw) u + Vu), \qquad T v_s = \Pi_-  e^{sw} (s (Tw) u + Tu),
\end{align*}
and since $u|_{\partial_-SM} = 0$, $Vu$ and $Tu$ will be controlled by $\|I_\Phi f\|_{H^1(\partial_+ SM)}$. The right hand side of \eqref{eq:partial} is thus bounded by $C' (s^2+s|\Phi|_{C^0}+1) e^{2sw_\infty} \|I_\Phi f\|^2_{H^1(\partial_+ SM)}$, where the constant $C'$ does not depend on $\Phi$.



Using this bound and throwing out the first and third terms of the left-hand side of \eqref{eq:partial}, we obtain
\begin{align*}
    \frac{\beta-1}{\beta} \|G_sVv_s\|^2 + \sum_{k=1}^\infty k |v_{-k}|^2 \le C' (s^2+s|\Phi|_{C^0}+1) e^{2sw_\infty} \|I_\Phi f\|^2_{H^1(\partial_+ SM)}.
\end{align*}
The second term in the left-hand side controls $\|v_s\|_{L^2}$ directly, and we can write
\begin{align}
    (\beta-1) \|G_sVv_s\|^2 + \|v_s\|^2 \le C' (s^2+s|\Phi|_{C^0}+1) e^{2sw_\infty} \|I_\Phi f\|^2_{H^1(\partial_+ SM)},
    \label{eq:estGvs}
\end{align}
with $C'$ some constant independent of $\Phi$. Recalling that $s = C |\Phi|_{C^1}+1$ and combining estimates \eqref{eq:estf}, \eqref{eq:est1}, \eqref{eq:estu0} and \eqref{eq:estGvs}, we arrive at estimate \eqref{eq:stabLinear}, completing the proof of Theorem \ref{thm:stabLinear}.

\subsubsection{Remaining ingredients}

\noindent {\bf Pestov identity with boundary term for ray transforms with skew-hermitian pairs}

Let $A$ and $\Phi$ a skew-hermitian pair, and define 
\begin{align*}
    G:= X+A+\Phi, \qquad G_\perp := X_\perp - A_V, \qquad \text{where } A_V := V(A).    
\end{align*}
We have the following structure equations
\begin{align}
    [G,V] = G_\perp, \qquad [V,G_\perp] = G-\Phi,\qquad [G,G_\perp] = -\kappa V - \star F_A - \star d_A\Phi, 
    \label{eq:structure}
\end{align}
where $\star d_A \Phi = X_\perp \Phi + \Phi A_V-A_V \Phi$, or when the connection $A$ is scalar, $\star d_A \Phi = X_\perp \Phi$, where $\kappa(x)$ is the Gaussian curvature. In what follows, we will need to integrate by parts with boundary terms, and using \eqref{eq:IBP}, we obtain for $G$:
\begin{align*}
    (Gu,v) = -(u,Gv) + (\mu u, v)_{\partial SM}.
\end{align*}

\begin{proof}[Proof of Theorem \ref{thm:pestovid}] We first write a differential identity using the structure equations \eqref{eq:structure}:
    \begin{align*}
	GVVG - VGGV &= GV[V,G] + [G,V]GV \\
	&= -GVG_\perp + G_\perp VG \\
	&= - G[V,G_\perp] + [G_\perp,G] V \\
	&= - G^2 + G\Phi + \kappa V^2 + \star F_A V + (\star d_A \Phi) V,
    \end{align*}
    where $G\Phi f := G(\Phi f)$. We record this here as 
    \begin{align}
	[GV,VG] = - G^2 + G\Phi + \kappa V^2 + \star F_A V + (\star d_A \Phi) V.
	\label{eq:Pestovdiff}
    \end{align}
    Now, considering $u$ smooth and supported up the boundary, we write
    \begin{align*}
	\|VGu\|^2 - \|GVu\|^2 &= (VGu,VGu) - (GVu,GVu) \\
	&= - (VVGu,Gu) + (GGVu,Vu) - (GVu, \mu\ Vu)_{\partial SM} \\
	&= ([GV,VG]u,u) - (VVGu, \mu\ u)_{\partial SM} - (GVu, \mu\ Vu)_{\partial SM} \\
	&= \|Gu\|^2 -(Gu,\mu\ u)_{\partial SM} - (\Phi u,Gu) + (\mu  \Phi u, u)_{\partial SM}\\
	&\qquad + (\kappa V^2u,u) + (\star F_A Vu,u) + ((\star d_A \Phi) Vu,u)\\
	&\qquad - (VVGu,\mu\ u)_{\partial SM} - (GVu, \mu\ Vu)_{\partial SM}
    \end{align*}    
    We now arrange the four boundary terms using integration by parts in $V$ and the formulas
    \begin{align*}
	V\mu = \langle v^\perp,\nu\rangle = \mu_\perp, \qquad V^2 \mu = V\mu_\perp =  - \mu.
    \end{align*}
    First notice that 
    \begin{align*}
	(VVGu,\mu\ u)_{\partial SM} &= - (VGu, (V\mu) u)_{\partial SM} - (VGu, \mu\ Vu)_{\partial SM}   \\
	&= - (VGu, \mu_\perp u)_{\partial SM} - (VGu, \mu\ Vu)_{\partial SM} \\ 
	&= - (Gu, \mu\ u)_{\partial SM} + (Gu, \mu_\perp Vu)_{\partial SM} - (VGu, \mu\ Vu)_{\partial SM}.
    \end{align*}
    We then obtain
    \begin{align*}
	(Gu,\mu\ u)_{\partial SM} &+ (VVGu, \mu\ u)_{\partial SM} + (GVu, \mu\ Vu)_{\partial SM} - (\mu\ \Phi u, u)_{\partial SM} \\
	&= (Gu, \mu_\perp\ Vu)_{\partial SM} - (VGu, \mu\ Vu)_{\partial SM}\\
	&\qquad + (GVu, \mu\ Vu)_{\partial SM} - (\mu\ \Phi u, u)_{\partial SM} \\
	&= (\mu_\perp\ Gu + \mu\ G_\perp u, Vu)_{\partial SM}- (\mu\ \Phi u, u)_{\partial SM}.
    \end{align*}
    We now simplify, using that $V(A)(x,v) = A(x,v^\perp)$ and $\mu_\perp X + \mu X_\perp = T$, 
    \begin{align*}
	\mu_\perp\ Gu + \mu\ G_\perp u = Tu + A(x,\nu^\perp) u + \mu_\perp \Phi u =: \nabla_{T,A} u + \mu_\perp \Phi u.
    \end{align*}
    The boundary terms then simplify into 
    \begin{align*}
	(\mu_\perp\ Gu &+ \mu\ G_\perp u, Vu)_{\partial SM}- (\mu\Phi u, u)_{\partial SM} \\
	&= (\nabla_{T,A}u, Vu)_{\partial SM} + (\mu_\perp\ \Phi u, Vu)_{\partial SM} - (\mu\ \Phi u, u)_{\partial SM} \\
	&= (\nabla_{T,A}u, Vu)_{\partial SM} + (\mu_\perp \Phi Vu, u)_{\partial SM}. 
    \end{align*}
    With this notation, the full Pestov identity takes the form
    \begin{align}
	\begin{split}
	    \|GVu\|^2 &- (Vu,\kappa Vu) + \|Gu\|^2 - \|VG u\|^2 - (\Phi u,Gu) \\
	    & + (\star F_A Vu, u) + ( (\star d_A\Phi)Vu,u) = (\nabla_{T,A}u, Vu)_{\partial SM} + (\mu_\perp\Phi Vu, u). 
	\end{split}
	\label{eq:Pestov}
    \end{align}
    To recover \cite[Eq. (8)]{PSUGAFA}, we take the real part of the equality above, and notice that $(\star F_A Vu, u) = -(\star F_A u, Vu)$ because $V(\star F_A) =0$; then 
    \begin{align*}
	((\star d_A \Phi) Vu, u) &= (V ((\star d_A \Phi)u), u) - (V (\star d_A \Phi) u, u) \\
	&= - ((\star d_A \Phi)u, Vu) - ( (d_A \Phi) u, u). 
    \end{align*}
    Since the last term is purely imaginary, the real parts of the other terms agree, and upon taking the real part of \eqref{eq:Pestov}, we obtain
    \begin{align}
	\begin{split}
	    \|GVu\|^2 &- (Vu,\kappa Vu) + \|Gu\|^2 - \|VG u\|^2 - \Re (\Phi u,Gu) \\
	    & - (\star F_A u, Vu) - \Re( (\star d_A\Phi)u, Vu) = \Re(\nabla_{T,A}u, Vu)_{\partial SM} + \Re( \mu_\perp \Phi Vu, u)_{\partial SM}. 
	\end{split}
	\label{eq:Pestov2}
    \end{align}
    (Note that the second boundary term is purely real so the $\Re$ is just ornamental)
    
    We finally explain how the index form term $\|GVu\|^2 - (Vu, \kappa Vu)$ can be rewritten as the sum of a non-negative term and a boundary term. With $\beta_\text{Ter}$ as in the statement, and the function $r = r_\beta\colon SM\to \mR$ solving $Xr + r^2 + \beta\kappa =0$, we now compute, for any $\psi\in C^\infty(SM,\C^n)$
    \begin{align*}
	\|G\psi - r\psi\|^2 = \|G\psi\|^2 - (G\psi, r\psi) - (r\psi, G\psi) + \|r\psi\|^2.
    \end{align*}
    We simplify
    \begin{align*}
	(G\psi,r\psi) + (r\psi,G\psi) &= (X\psi, r\psi) + (r\psi, X\psi) \\
	&= \int_{SM} (X\psi) r\bar{\psi} + r\psi (X\bar{\psi}) \\
	&= \int_{SM} X (\psi r\bar{\psi}) - (Xr) \psi\bar{\psi} \\
	&= (\mu\ r\psi, \psi)_{\partial SM} + \int_{SM} (r^2 + \beta\kappa) \psi \bar{\psi}.
    \end{align*} 
    We arrive at 
    \begin{align*}
	\|G\psi - r\psi\|^2 = \|G\psi\|^2 - (\mu r\psi, \psi)_{\partial SM} - \beta (\kappa \psi,\psi),
    \end{align*}
    and we may rearrange this as 
    \begin{align*}
	\beta (\|G\psi\|^2 - (\kappa\psi,\psi)) = \|G\psi-r\psi\|^2 + (\beta-1) \|G\psi\|^2 + (\mu\ r\psi, \psi)_{\partial SM}.
    \end{align*}
    Plugging this last relation into \eqref{eq:Pestov2} with $\psi = Vu$ yields \eqref{eq:Pestov3}.
\end{proof}

\noindent {\bf Remaining estimates and lemmata} \label{sec:miscids}

\begin{proof}[Proof of equality \eqref{eq:temp}]
    We write, using Lemma \ref{lem:HIF}
    \begin{align*}
	\eta_+ u_{-1} &= \eta_+ (e^{-sw} \Pi_- (e^{sw}u))_{-1} \\
	&= \eta_+ \left[ \sum_{k=0} (e^{-sw})_{2k} (e^{sw}u)_{-1-2k} \right] \\
	&= \sum_{k=0}^{\infty} \left( ( (\eta_+-sa_1) (e^{-sw})_{2k}) (e^{sw}u)_{-1-2k} + (e^{-sw})_{2k} (\eta_++sa_1) (e^{sw}u)_{-1-2k} \right) \\
	&= ((e^{sw}u) (\eta_+ - sa_1) (e^{-sw}))_0 + \sum_{k=0}^\infty (e^{-sw})_{2k} (\eta_++sa_1) (e^{sw}u)_{-1-2k}.
    \end{align*}
    To rewrite the last term, from the equation $G_s (e^{sw}u) = -e^{sw}f$, note the relation
    \begin{align*}
	(\eta_++sa_1) (e^{sw}u)_{-1-2k} + (\eta_{-}+sa_{-1}) (e^{sw}u)_{1-2k} + \Phi (e^{sw}u)_{-2k} = 0.
    \end{align*}
    Then we have, for $k>0$,
    \begin{align*}
	(G_s V \Pi_- (e^{sw}u))_{-2k} &= \underbrace{V (G_s \Pi_- (e^{sw}u))_{-2k}}_{=0} + ([G_s, V] \Pi_- (e^{sw}u))_{-2k} \\
	&= -i (\eta_+ + sa_1) (e^{sw}u)_{-1-2k} + i (\eta_{-}+sa_{-1}) (e^{sw}u)_{1-2k} \\
	&= -2i (\eta_+ + sa_1) (e^{sw}u)_{-1-2k} -i \Phi (e^{sw}u)_{-2k},
    \end{align*}
    where we used the transport equation in the last line. For $k=0$, 
    \begin{align*}
	(G_s V \Pi_- (e^{sw}u))_{0} = -i (\eta_+ + s a_1) (e^{sw}u)_{-1}.
    \end{align*}
    Plugging this back into the equation for $\eta_+ u_{-1}$, we get 
    \begin{align*}
	\eta_+ u_{-1} &= ((e^{sw}u) (\eta_+ - sa_1) (e^{-sw}))_0 + (e^{-sw})_0 i (G_s V\Pi_- (e^{sw}u))_{0} \\
	& \qquad + \sum_{k>0} (e^{-sw})_{2k} \left( \frac{i}{2} (G_s V\Pi_- (e^{sw}u))_{-2k} - \frac{1}{2} \Phi (e^{sw}u)_{-2k}\right) .
    \end{align*}
    We now write
    \begin{align*}
	\sum_{k>0} (e^{-sw})_{2k}\Phi (e^{sw}u)_{-2k} &= \sum_{k=0}^\infty (e^{-sw})_{2k}\Phi (e^{sw}u)_{-2k} - e^{-sw_0} \Phi (e^{sw}u)_{0} \\
	&= \Phi u_0 - e^{-sw_0} \Phi (e^{sw}u)_{0}
    \end{align*}
    and similarly 
    \begin{align*}
	\sum_{k>0} (e^{-sw})_{2k} (G_s V\Pi_- (e^{sw}u))_{-2k} = (e^{-sw} G_s V\Pi_- (e^{sw}u))_0 - e^{-sw_0} (G_s V\Pi_- (e^{sw}u))_{0}.
    \end{align*}
    Using the last two computations, we arrive at \eqref{eq:temp}.    
\end{proof}

\begin{proof}[Proof of estimate \eqref{eq:estu0}] The transport equation for $e^{sw}u$ projected onto the harmonic term of degree $-1$ reads: 
    \begin{align*}
	(\eta_- + sa_{-1}) (e^{sw}u)_0 = - (\eta_+ + sa_1) (v_s)_{-2} - \Phi (v_s)_{-1}. 
    \end{align*}
    For our choice of connection, $a_{-1} = -\eta_- w_0$ so the left side can be rewritten as 
    \begin{align*}
	(\eta_- + sa_{-1}) (e^{sw}u)_0  = e^{sw_0} \eta_- (e^{-sw_0} (e^{sw}u)_0) = e^{sw_0} \eta_- (e^{s(w-w_0)}u)_0,
    \end{align*}
    hence we obtain 
    \begin{align*}
	\eta_- (e^{s(w-w_0)}u)_0 = - e^{-sw_0} (\eta_+ + sa_1) (v_s)_{-2} - e^{-sw_0} \Phi (v_s)_{-1}.
    \end{align*}
    We then rewrite the latter right-hand side in terms of $G_sV v_s$. Notice that 
    \begin{align*}
	(G_s V v_s)_{-1} &= (\eta_+ + sa_1) (Vv_s)_{-2} + \Phi (V v_s)_{-1} \\
	&= -2i (\eta_+ + sa_1) (v_s)_{-2} - i \Phi (v_s)_{-1},
    \end{align*} 
    so 
    \begin{align*}
	-(\eta_+ + sa_1) (v_s)_{-2} = -\frac{i}{2} (G_s V v_s)_{-1} + \frac{1}{2} \Phi (v_s)_{-1},
    \end{align*}
    and thus 
    \begin{align*}
	\eta_- (e^{s(w-w_0)}u)_0 = - \frac{e^{-sw_0}}{2} \left( (G_s V v_s)_{-1} + \Phi (v_s)_{-1} \right)
    \end{align*}
        Upon deriving an estimate of the form 
    \begin{align*}
	\|f\|_{L^2(M)} \le C (\|\eta_- f\|_{L^2(M)} + \|f|_{\partial M}\|_{L^2(\partial M)}), 
    \end{align*}
    we can write 
    \begin{align*}
	\|(e^{s(w-w_0)}u)_0\|_{L^2(M)} &\lesssim \|\eta_- (e^{s(w-w_0)}u)_0\|_{L^2(M)} + \|(e^{s(w-w_0)}u)_0|_{\partial M}\|_{L^2(\partial M)} \\
	&\lesssim \frac{1}{2} \|e^{-sw_0} ((i G_s Vv_s)_{-1}+\Phi (v_{s})_{-1} \|_{L^2(M)} \\
	& \qquad\qquad+ \|(e^{s(w-w_0)}u)_0)|_{\partial M}\|_{L^2(\partial M)} ,
    \end{align*}
    and \eqref{eq:estu0} follows.     
\end{proof}

\begin{proof}[Proof of \eqref{eq:tedious}] We first need to write an integration by parts for $\mu_\pm$ defined in \eqref{eq:etapm}. Using integrations by parts \eqref{eq:IBP} we first derive an integration by parts for $X_\perp = XV-VX$: for any $u,w$ smooth on $SM$,
    \begin{align*}
	(X_\perp u, w) + (u,X_\perp w) &= (XV u, w) - (VXu, w) + (u,XVw) - (u,VXw) \\
	&= (XVu,w) + (Vu,Xw) + (Xu,w) + (u,XVw) \\
	&= (\mu Vu, w)_{\partial SM} + (\mu u, Vw)_{\partial SM} \\
	&= - ( (V\mu) u, w)_{\partial SM}  = -(\mu_\perp u,w )_{\partial SM}
    \end{align*}
    We now compute, using that $\mu_+^* = -\mu_-$ 
    \begin{align*}
	(u, \eta_+ w) + (\eta_- u, w) &= \frac{1}{2} \left( (u, (X+iX_\perp)w) + (X-iX_\perp u,w) \right) \\
	&=  \frac{1}{2} \left( (\mu + i\mu_\perp) u, w    \right)_{\partial SM} = ( e_x(v) u, w)_{\partial SM}
    \end{align*} 
    where we define
    \begin{align}
	e_x(v) := \frac{1}{2} (\mu(x,v) + i \mu_\perp(x,v)).
	\label{eq:exv}
    \end{align}
    Similarly, for the skew-hermitian connection considered, 
    \begin{align*}
	(u, (\eta_+ + sa_1) w) + ((\eta_- + sa_{-1})u, w) = ( e_x(v) u, w)_{\partial SM}.
    \end{align*}    
    Now, using the fact that 
    \begin{align*}
	(G_s v_s)_{-1} = (\eta_{-} + sa_{-1}) (e^{sw}u)_0 = -(\eta_+ + sa_1)(v_s)_{-2} - \Phi (v_s)_{-1},
    \end{align*}
    we compute
    \begin{align*}
	\Re (\Phi v_s,G_s v_s) &= \Re(\Phi (v_s)_{-1}, (G_s v_s)_{-1}) \\
	&= \Re(\Phi (v_s)_{-1}, -(\eta_+ + sa_1)(v_s)_{-2}) - |\Phi (v_s)_{-1}|^2 \\
	&= \Re ( (\eta_-+sa_{-1})(\Phi (v_s)_{-1}), (v_s)_{-2} )  \\
	& \qquad\qquad\qquad - \Re(e_x(v)\Phi(v_s)_{-1}, (v_s)_{-2} )_{\partial SM} - |\Phi (v_s)_{-1}|^2 \\
	&= - |\Phi (v_s)_{-1}|^2 + \Re (b_{-1} (v_s)_{-1}, (v_s)_{-2})  \\
	& \qquad\qquad\qquad - \Re(e_x(v)\Phi(v_s)_{-1}, (v_s)_{-2} )_{\partial SM} + p_1,
    \end{align*}
    where $p_1:= \Re (\Phi (\eta_- + sa_{-1}) (v_s)_{-1}, (v_s)_{-2})$. Upon defining 
    \begin{align}
	p_n:= \Re (\Phi (\eta_- + sa_{-1}) (v_s)_{-n}, (v_s)_{-n-1}), \qquad n\ge 1,
	\label{eq:pn}
    \end{align}
    we now prove by induction the following claim: 
    \begin{align}
	\begin{split}
	    \Re (\Phi v_s,G_s v_s) &= \sum_{k=1}^n (-1)^k \Big( |\Phi (v_s)_{-k}|^2 - \Re (b_{-1} (v_s)_{-k}, (v_s)_{-k-1}) \\
	    &\qquad \qquad + \Re (e_x(v)\Phi (v_s)_{-k}, (v_s)_{-k-1})_{\partial SM}  \Big) + (-1)^{n+1} p_n.	
	\end{split}    
	\label{eq:secondterm}
    \end{align}
    The case $n=1$ is proved above, and the induction step $(n\implies n+1)$ follows from the calculation
    \begin{align*}
	p_n &= \Re (\Phi (\eta_- + sa_{-1}) (v_s)_{-n}, (v_s)_{-n-1}) \\
	&= - \Re (\Phi(v_s)_{-n-1}, (\eta_- + sa_{-1}) (v_s)_{-n}) \\
	&= \Re (\Phi(v_s)_{-n-1}, (\eta_+ + sa_1) (v_s)_{-n-2} + \Phi (v_s)_{-n-1} ) \\
	&= |\Phi(v_s)_{-n-1}| + \Re (\Phi(v_s)_{-n-1}, (\eta_+ + sa_1) (v_s)_{-n-2})_{\partial SM} \\
	&= |\Phi(v_s)_{-n-1}| - \Re ( (\eta_- + sa_{-1}) (\Phi(v_s)_{-n-1}),  (v_s)_{-n-2}) \\
	& \qquad\qquad\qquad + \Re (e_x(v) \Phi(v_s)_{-n-1}, (v_s)_{-n-2} )_{\partial SM} \\
	&= |\Phi(v_s)_{-n-1}| - \Re (b_{-1} (\Phi(v_s)_{-n-1}),  (v_s)_{-n-2}) \\
	& \qquad\qquad\qquad + \Re (e_x(v) \Phi(v_s)_{-n-1}, (v_s)_{-n-2} )_{\partial SM} - p_{n+1}.
    \end{align*}
    Putting this equality back into \eqref{eq:secondterm} proves the induction. Now since $v_s\in H^1(SM)$, we have that $\lim_{n\to \infty} p_n = 0$, and thus \eqref{eq:tedious} follows.
\end{proof}

\begin{proof}[Proof of Lemma \ref{lem:s}] The term that ultimately controls everything is 
    \begin{align*}
	s(v_s, iVv_s) = s\sum_{k<0} |k| |(v_s)_k|^2.
    \end{align*}
    The infinite sum in \eqref{eq:sbound} can then be controlled by 
    \begin{align*}
	\sum_{k=1}^\infty (-1)^k \left( |\Phi (v_s)_{-k}|^2 - \Re (B_{-1} (v_s)_{-k}, (v_s)_{-k-1})\right) \le C_1 |\Phi|_{C^1} \sum_{k<0} |(v_s)_k|^2,
    \end{align*}
    with $C_1$ a universal constant. As for the last term of the left-hand side of \eqref{eq:sbound}, we write
    \begin{align*}
	((X_\perp\Phi)v_s, Vv_s) &= ((-iB_1 + i B_{-1}) v_s, Vv_s)\\
	&= (B_1 v_s - B_{-1}v_s, iV v_s) \\
	&= \sum_{k<0} k(B_1 (v_s)_{k-1} - B_{-1} (v_s)_{k+1}, (v_s)_k) \\
	|((X_\perp\Phi)v_s, Vv_s)| &\le C_2 |\Phi|_{C^1} \sum_{k<0} |k| |(v_s)_k|^2,
    \end{align*}
    where $C_2$ is a universal constant. Lemma \ref{lem:s} follows upon taking $C = C_1+C_2$.
\end{proof}

\subsubsection{Conclusion: dealing with the glancing} Consider a function $\rho\in C^{\infty}(M)$ such that it coincides with $M\ni x\mapsto d(x,\partial M)$ in a neighbourhood of $\partial M$ and such that $\rho\geq 0$ and $\partial M=\rho^{-1}(0)$.
Clearly $\nabla\rho(x)=-\nu(x)$ for $x\in \partial M$. Using $\rho$, we extend $\nu$ to the interior of $M$ as $\nu(x)=-\nabla\rho(x)$ for $x\in M$. We let $\mu(x,v):=\langle v,\nu(x)\rangle$ and
\[T:=V(\mu)X+\mu X_{\perp}.\]
Note that $T$ is now defined on all $SM$ and agrees with the vector field $T$ defined previously on $\partial SM$. In fact $T$ and $V$ are tangent to every $\partial SM_{\varepsilon}=\{(x,v)\in SM:\;\;x\in \rho^{-1}(\varepsilon)\}$,
where $M_{\varepsilon}=\rho^{-1}(-\infty,\varepsilon]$. The next lemma for $\tau$ is the key input to deal with the glancing, cf. \cite[Lemma 4.1.3]{Sharafutdinov}, \cite[Lemma 3.2.3]{Sharafut99} and \cite[Lemma 5.1]{DPSU07}.

\begin{Lemma} The functions $V\tau$ and $T\tau$ are bounded on $SM\setminus\partial_{0}SM$.
\label{lemma:tau}
\end{Lemma}

To substantiate the previous claim that the behaviour of $u=u^f$ is the same as that of $\tau$ we proceed as follows.
We consider a smooth integrating factor $R:SM\to GL(n,\C)$ such that $XR+\Phi R=0$.
These always exist for any non-trapping manifold with strictly convex boundary. A simple calculation shows that we may write $u$ in terms
of $R$ as
\[u(x,v)=R(x,v)\int_{0}^{\tau(x,v)}(R^{-1}f)(\varphi_{t}(x,v))\,dt\;\;\text{for}\;(x,v)\in SM,\]
where $\varphi_t$ is the geodesic flow of $(M,g)$.
Thus directly from Lemma \ref{lemma:tau} we obtain:

\begin{Lemma} The functions $Vu$ and $Tu$ are bounded on $SM\setminus\partial_{0}SM$.
\label{lemma:u}
\end{Lemma}


Next we note that all the previous work that we have done assuming $u$ smooth may be summarized as follows:

\begin{Theorem} \label{thm:stabLinearT} Let $(M,g)$ be a simple Riemannian surface with boundary and $\Phi$ a smooth, skew-hermitian matrix field in $M$. Then for any $f\in C^\infty(M)$, we have the following stability estimate
    \begin{align*}
	\|f\|_{L^2(M,\C^n)} \le C_1 (1+\|\Phi\|_{C^1}) e^{C_2 \|\Phi\|_{C^1}} \|v\|_{H^1(\partial SM, \C^n)},
    \end{align*}        
where $v$ is any smooth solution of $Xv+\Phi v=-f$.
\end{Theorem}

\begin{proof}[Proof of Theorem \ref{thm:stabLinear} in full generality] Let $M_{\varepsilon}$ for small $\varepsilon$ be the surface considered above. We let $u:SM\to \C^n$ be the unique solution to the problem
\begin{align*}
    Xu + \Phi u = -f \qquad (SM), \qquad u|_{\partial_- SM} = 0.
\end{align*}
The function $v:=u|_{SM_{\varepsilon}}$ is smooth in $SM_{\varepsilon}$ and solves $Xv+\Phi v=-f$
since $u$ does. Hence we may apply Theorem \ref{thm:stabLinearT} in $M_{\varepsilon}$ to obtain
\begin{align*}
	\|f\|_{L^2(M_{\varepsilon},\C^n)} \le C_1 (1+\|\Phi\|_{C^1}) e^{C_2 \|\Phi\|_{C^1}} \|u\|_{H^1(\partial SM_{\varepsilon}, \C^n)},
    \end{align*}   
where we might as well use the constants for $M$ which bound those for $M_{\varepsilon}$.
We now let $\varepsilon\to 0$; we clearly have
 \[\|f\|_{L^2(M_{\varepsilon},\C^n)} \to \|f\|_{L^2(M,\C^n)}\]
and using Lemma \ref{lemma:u} we see that
\[\|u\|_{H^1(\partial SM_{\varepsilon}, \C^n)}\to \|u\|_{H^1(\partial SM, \C^n)}.\]
Since 
\[
u(x,v)= \left\{ \begin{array}{cl} I_{\Phi}(f)(x,v), &(x,v)\in \partial_{+}SM, \\[5pt] 0, & (x,v)\in\partial_{-}SM, \end{array} \right.
\]
the theorem is proved.

\end{proof}

\subsection{Consistency of the posterior mean: proof of Theorem \ref{main}}\label{prfbayes}
 
 We assume $\sigma^2=1$, the general case $0<\sigma^2<\infty$ requires only notational changes.
 
 \smallskip
 
The overall strategy we pursue here, which has also been used in some form in \cite{V13, NS17, N1, NvdGW18}, is to show first that the Bayesian algorithm recovers the `regression function' $C_\Phi$ consistently in a natural statistical distance function, and to combine this with quantitative \textit{stability estimates} for the inverse map $C_\Phi \mapsto \Phi$ in appropriate metrics. This exploits crucially that the estimated Bayesian regression outputs lie in the (non-linearly constrained) range of the forward map $C_\Phi$, so that the stability estimate applies to them. To make this approach work with `unbounded' Gaussian priors is challenging, and our proofs proceed as follows:  We first establish the posterior contraction Theorem \ref{general} under general conditions, borrowing from Bayesian nonparametric theory (e.g.,  \cite[Theorem 8.19]{GvdV17} or \cite[Theorem 7.3.3]{GN16}), slightly strengthening the usual statement of such theorems to give explicit exponential bounds for the convergence rate to zero of certain posterior probabilities. Since our regression functions $C_\Phi$ take values in $SO(n)$, they are uniformly bounded and the usual Hellinger distance occurring in such contraction theorems is then Lipschitz-equivalent to the standard $L^2$-distance (see Lemma \ref{birge}). Then Lemma \ref{smallv} uses results of \cite{LL99} to show that the key small ball condition in Theorem \ref{general} can be verified for the Gaussian priors from Condition \ref{pco} even after they have been shrunk towards zero, if the true matrix field $\Phi_0$ belongs to the RKHS $\mathcal H$. Next, Lemma \ref{regularity} exploits fine properties \cite{B75, F75} of infinite-dimensional Gaussian measures to show that such `shrunk' priors charge `sufficiently regular' matrix fields (effectively $C^\beta$-balls) with probability close enough to one that the posterior distributions inherits these regularity properties. This is crucial to apply the `forward' estimate Theorem \ref{thm:forward} and the `stability' estimate (\ref{stabi}) in the proof of Theorem \ref{overall} -- effectively the specific structure of our inverse problem enters only in this theorem and only through these two estimates. Finally, the exponential convergence to zero of the order $e^{-(C+3)N \delta_N^2}$ obtained in Theorem \ref{overall} permits a `quantitative uniform integrability argument' in Section \ref{uiq} to deduce convergence of the whole posterior (Bochner-) mean towards the true matrix field $\Phi_0$.  

\smallskip

Let us mention that in the recent contributions \cite{AN20, GN20} (written after the first version of this manuscript was completed), the general proof template developed here has already been used effectively in two different non-linear inverse problems (arising with elliptic PDEs), see also Remark \ref{genrem}.

 \subsubsection{A general contraction theorem}

 Consider a collection $\mathcal P$ of probability density functions on some measurable space $(\mathcal X, \mathcal A)$ with respect to a dominating measure $\mu$, specifically in our measurement model (\ref{model}) we take $$\mathcal P = \Big\{p_\Phi= \frac{dP^1_\Phi}{d\mu}: \Phi \in C(M, \so(n))\Big\}, ~~\mathcal X = \mathbb R^{n \times n} \times \partial_+ SM,$$ where $\mathcal X$ is equipped with its natural product Borel-$\sigma$ algebra $\mathcal A$, where $d\mu = dy \times d\lambda$ with $dy$ equal to Lebesgue measure on $\mathbb R^{n \times n}$ and $\lambda$ given in (\ref{lambo}). By the Gaussianity of the $\varepsilon_{1,j,k}$'s these probability densities are of the form
 \begin{equation}\label{moddens}
 p_\Phi(y, (x,v)) = \frac{1}{(2\pi)^{n^2/2}} \exp \Bigg\{-\frac{1}{2}\sum_{1 \le j,k \le n}\big[y_{j,k} -(C_\Phi((x,v))_{j,k}) \big]^2\Bigg\},~(y, (x,v)) \in \mathcal X.
 \end{equation}
 Since the map $(\Phi, y,(x,v)) \mapsto p_\Phi(y, (x,v))$ is jointly Borel-measurable from $C(M) \times \mathcal X$ to $\mathbb R$ (using (\ref{eq:continf}) and that point-evaluation is $\|\cdot\|_\infty$-continuous), the posterior distribution (\ref{post}) exists by standard arguments (\cite{GvdV17}, p.7) and has the desired form. In the proof of the following theorem we show in particular that the marginal density $\int \prod_{i=1}^N p_\Phi(Y_i, (X_i, V_i))d\Pi(\Phi)$ is positive on events of $P_{\Phi_0}^N$-probability approaching one, so that (\ref{post}) is well-defined also in the frequentist setting where $D_N \sim P_{\Phi_0}^N$. We also define the Hellinger distance $h$ on such densities by
 \begin{equation}\label{hellinger}
 h^2(p_\Phi, p_\Psi) = \int_\mathcal X \big(\sqrt {p_\Phi} - \sqrt {p_\Psi} \big)^2 d\mu,~~ \Phi, \Psi \in C(M, \so(n)).
 \end{equation}
Denote by $N(F, h, \delta)$ the minimal number of Hellinger-balls of radius $\delta$ required to cover a set $F$ of $\mu$-densities on $\mathcal X$. We then have the following
\begin{Theorem}\label{general}
Consider a prior for $\Phi$ arising from a sequence $\Pi=\Pi_N$ of Borel probability measures on $\mathcal F \subseteq C(M, \so(n))$ and let  $\Pi(\cdot|(Y_i, (X_i,V_i))_{i=1}^N)$ be the posterior distribution arising from i.i.d.~observations $(Y_i, (X_i,V_i))_{i=1}^N|\Phi \sim P_\Phi^N$. Let $\Phi_0 \in \mathcal F$, let $\delta_N \to 0$ be a sequence such $\sqrt N \delta_N \to \infty$ as $N \to \infty$, and define sets 
\begin{equation} \label{klnbhd}
B_N=\Big\{\Phi \in \mathcal F: E^1_{\Phi_0} \Big[\log \frac{p_{\Phi_0}}{p_\Phi}((Y, (X,V))\Big] \le \delta_N^2,\ E^1_{\Phi_0}\Big[ \log \frac{p_\Phi}{p_{\Phi_0}}((Y, (X,V))\Big]^2 \le \delta_N^2\Big\}.
\end{equation}
 Suppose for some constant $C>0$ the prior $\Pi$ satisfies for all $N$ large enough
\begin{equation}\label{smallball}
\Pi(B_N) \ge e^{-CN\delta_N^2},
\end{equation}
and that for some sequence $\mathcal F_N \subset \mathcal F$ of approximating sets  for which
\begin{equation} \label{excessmass}
\Pi\big(\mathcal F \setminus \mathcal F_N \big)\le L e^{-(2C+6)N \delta_N^2},\quad \text{for some } 0<L<\infty,
\end{equation}
we have the complexity bound
\begin{equation}\label{ment}
\log N(\mathcal F_N, h, \delta_N)  \le cN \delta_N^2,
\end{equation}
for some fixed constant $c>0$. Then for some large enough constant $m=m(C,c)>0$
\begin{equation}
P_{\Phi_0}^N \left(\Pi\big(\mathcal F_N\cap \{\Phi: h(p_\Phi, p_{\Phi_0}) \le m \delta_N\} | (Y_i, (X_i,V_i))_{i=1}^N\big) \le 1- e^{-(C+3)N \delta_N^2} \right) \to_{N \to \infty} 0.
\end{equation}
\end{Theorem} 
\begin{proof}
    Recall from (\ref{data}) that we write $D_N =(Y_i, (X_i,V_i))_{i=1}^N$. The proof proceeds as in the proof of \cite[Theorems 7.3.1 and 7.3.3]{GN16}: We first use \cite[Lemma 7.3.2]{GN16} and the hypothesis \eqref{smallball} to deduce that the events 
\begin{equation} \label{anset}
    A_N = \left\{\int_\mathcal F \prod_{i=1}^N \frac{p_\Phi}{p_{\Phi_0}}(Y_i, (X_i, V_i))\ d\Pi(\Phi) \ge e^{-(2+C)N\delta_N^2}\right\}
\end{equation}
satisfy $P_{\Phi_0}^N(A_N) \to 1$ as $N \to \infty$. Moreover using (\ref{ment}) and \cite[Theorem 7.1.4]{GN16} with choices $\varepsilon_0 = m' \delta_N$, any $m'<m$, and $\log N(\varepsilon)=c N\delta_N^2$ constant in $\varepsilon>\varepsilon_0$, we deduce that for every $k>1$ there exists $m', m$ large enough such that we can find `tests' (random indicator functions) $\Psi_N=\Psi_N(D_N)$ for which 
\begin{align}\label{tests}
P_{\Phi_0}^N(\Psi_N =1) \to_{N \to \infty} 0 \text{ and } \sup_{\Phi \in \mathcal F_N: h(p_\Phi, p_{\Phi_0}) > m \delta_N} E_\Phi^N(1-\Psi_N) \le e^{-kN\delta_N^2}.
\end{align}
Now let us write $$\bar F_N = \mathcal F_N\cap \{h(p_\Phi, p_{\Phi_0}) \le m \delta_N\}$$ for the event whose posterior probability we want to bound. Then by (\ref{post}) and as $N \to \infty$,
\begin{align*}
&P_{\Phi_0}^N \big(\Pi\big(\bar F^c_N | D_N\big) \ge  e^{-(C+3)N \delta_N^2}  \big) \\
&= P_{\Phi_0}^N \left(\frac{\int_{\bar F_N^c} \prod_{i=1}^N \frac{p_\Phi}{p_{\Phi_0}}(Y_i, (X_i, V_i))d\Pi(\Phi)}{\int_\mathcal F \prod_{i=1}^N \frac{p_\Phi}{p_{\Phi_0}}(Y_i, (X_i, V_i))d\Pi(\Phi)} \ge  e^{-(C+3)N \delta_N^2} , \Psi_N=0, A_N \right) +o(1)\\
&\le P_{\Phi_0}^N \left(\int_{\bar F_N^c} \prod_{i=1}^N \frac{p_\Phi}{p_{\Phi_0}}(Y_i, (X_i, V_i))d\Pi(\Phi)(1-\Psi_N) \ge  e^{-(2C+5)N \delta_N^2} \right) +o(1).
\end{align*}
By Markov's inequality, decomposing $$\bar F_N^c = \mathcal F_N^c \cup \{h(p_\Phi, p_{\Phi_0}) > m \delta_N\},$$ and using Fubini's theorem as well as 
\begin{equation} \label{com}
E_{\Phi_0}^N\prod_{i=1}^N \frac{p_\Phi}{p_{\Phi_0}}(Y_i, (X_i, V_i))(1-\Psi_N) = E_\Phi^N(1-\Psi_N) \le 1
\end{equation}
 we further bound the last probability as
\begin{align*}
& e^{(2C+5)N\delta_N^2} \int_{\bar F_N^c} E_{\Phi}^N(1-\Psi_N)d\Pi(\Phi) \\
&\le e^{(2C+5)N\delta_N^2}\left(2\Pi(\mathcal F_N^c) +  \int_{\Phi \in \mathcal F_N:h(p_\Phi, p_{\Phi_0}) > m \delta_N}E_{\Phi}^N(1-\Psi_N)d\Pi(\Phi) \right) \\
&\le 2Le^{-N\delta_N^2} + e^{(2C+5-k)N\delta_N^2} \to_{N \to \infty} 0,
\end{align*}
where we have used (\ref{excessmass}) and (\ref{tests}) with $k$ and then $m$ large enough.
\end{proof}

The `information-theoretic distance' $h$ arises naturally in such posterior contraction theorems, see \cite{GvdV17}. The following lemma, which adapts a result due to Birg\'e \cite{B04} to the setting of $SO(n)$-valued functions, shows that the Hellinger distance is Lipschitz equivalent to the standard $L^2$-metric $$\|C_{\Phi}-C_\Psi\|_{L^2} =\sqrt{\sum_{1 \le j,k \le n}\|C_{\Phi,j,k}-C_{\Psi,j,k}\|^2_{L^2}}.$$

\begin{Lemma}\label{birge}
For $\Phi \in C(M, \so(n))$, let $C_\Phi\colon \partial_+ SM  \to SO(n)$ be its non-Abelian $X$-ray transform. Then there exist positive constants $c_0=c_0(n), c_1=c_1(n)$ such that
\begin{equation}\label{infeq}
\frac{1}{c_0}\|C_{\Phi}-C_\Psi\|_{L^2}^2 \le h^2(p_\Phi, p_\Psi) \le  c_1\|C_\Phi-C_\Psi\|_{L^2}^2, \quad \forall ~\Phi, \Psi \in C(M, \so(n)).
\end{equation}
\end{Lemma} 
\begin{proof}
Write 
\begin{equation} \label{affin}
\rho(p_\Phi, p_\Psi) \equiv \int_\mathcal X \sqrt{p_\Phi p_{\Psi}} d\mu=1-\frac{1}{2}h^2(p_\Phi, p_\Psi)
\end{equation}
for the `Hellinger affinity'. By (\ref{moddens}) and using the standard formula for the moment generating function of $N(0,1)$-variables, the quantity $\rho(p_\Phi, p_\Psi)$ equals
\begin{align*}
&=\frac{1}{(2\pi)^{n^2/2}} \int_\mathcal X \!\!\! \exp \Big\{\frac{1}{4}\sum_{j,k}\Big[-\big[y_{j,k} -(C_\Phi((x,v))_{j,k})\big]^2 -\big[y_{j,k} -(C_\Psi((x,v))_{j,k}) \big]^2\Big]\Big\} \\
& = \int_{\partial_+SM} \exp\Big\{-\frac{1}{4}\sum_{j,k}\big[C^2_\Phi((x,v))_{j,k} +C^2_\Psi((x,v))_{j,k}\big]  \Big\} \\
& \quad \times \frac{1}{(2\pi)^{n^2/2}}\int_{\mathbb R^{n \times n}} e^{\frac{1}{2} \sum_{j,k} y_{j,k}(C_\Phi((x,v))_{j,k} +C_\Psi((x,v))_{j,k})} e^{-\sum_{j,k}y_{jk}^2/2} dy d\lambda(x,v)  \\
&= \int_{\partial_+SM} \exp\Big\{-\frac{2}{8}\sum_{j,k}\big[C^2_\Phi((x,v))_{j,k} +C^2_\Psi((x,v))_{j,k}\big]\Big\}  \\ 
&\quad \quad \quad \times \exp\Big\{\frac{1}{8} \sum_{j,k}\big[C_\Phi((x,v))_{j,k} +C_\Psi((x,v))_{j,k}\big]^2 \Big\} d\lambda(x,v) \\
&= \int_{\partial_+SM}  \exp\Big\{-\frac{1}{8} \frob{C_\Phi(x,v)-C_\Psi(x,v)}^2 \Big\} d\lambda(x,v).
\end{align*}
By Jensen's inequality the last integral is greater than or equal to $\exp\{-\|C_\Phi-C_\Psi\|_{L^2}^2/8\}$ and using standard inequalities for $1-e^{-z}, z>0,$ the right hand side of (\ref{infeq}) follows. Next we notice that since $C_\Phi(x,v), C_\Psi(x,v) \in SO(n)$, their matrix entries are all bounded by one and we hence have $\frob{C_\Phi(x,v)-C_\Psi(x,v)}^2/8 \le B^2$ for some constant $B=B(n)$. We can thus proceed exactly as in the proof of \cite[Proposition 1]{B04} (or see Lemma 22 in \cite{GN20}) to also deduce the left hand side inequality in (\ref{infeq}).
\end{proof}

 \subsubsection{Verification of the prior mass condition}

We now verify condition (\ref{smallball}) in the last theorem for an explicit constant $C>0$ and the Gaussian prior from Theorem \ref{main}. To do this we first show that one can reduce to checking small ball conditions for $\|\cdot\|_{L^2(M)}$-norms on the level of the original matrix parameter $\Phi$.

\begin{Lemma}\label{klrid}
For $\Phi_0 \in C(M, \so(n))$ and $\kappa>0$ define $$\mathcal B_N(\kappa) = \{\Phi \in C(M, \so(n)): \|\Phi - \Phi_0\|_{L^2(M)} \le \delta_N/\kappa\}$$ and let $B_N, \Pi, \delta_N,$ be as in Theorem \ref{general}. Then for some $\kappa=\kappa(M,n)$ large enough we have $\mathcal B_N(\kappa) \subset B_N$ and thus in particular, for every $N \in \mathbb N$, $$\Pi(B_N) \ge \Pi(\mathcal B_N(\kappa)).$$
\end{Lemma}
\begin{proof}
From (\ref{model}) with $\Phi=\Phi_0$ and (\ref{moddens}) we have 
\begin{align*}
& \log p_\Phi (Y_1, (x,v))  - \log p_{\Phi_0}(Y_1, (x,v)) = \\
& -\sum_{1 \le j,k \le n}\left[\frac{1}{2}(C_\Phi((x,v))_{j,k}-C_{\Phi_0}((x,v))_{j,k})^2 + \varepsilon_{1,j,k} (C_\Phi((x,v))_{j,k}-C_{\Phi_0}((x,v))_{j,k})\right].
\end{align*}
Therefore, since $E^1_\varepsilon \varepsilon_{1,j,k}=0$ and $\lambda$ is the unit volume measure on $\partial_+SM$,
\begin{align*}E^1_{\Phi_0} \Big[\log \frac{p_{\Phi_0}}{p_\Phi}((Y, (X,V))\Big] &= \frac{1}{2}\|C_\Phi - C_{\Phi_0}\|^2_{L^2(\partial_+ SM)}  \le \frac{C_1^2}{2}\|\Phi-\Phi_0\|^2_{L^2(M)},
\end{align*}
where we have used the forward estimate (\ref{eq:cont0}). Thus if $\kappa \ge 2/C_1^2$ the first inequality defining $B_N$ is verified for $\Phi \in \mathcal B_N(\kappa)$. 
To verify the second, note that all $C_\Phi(x,v) \in SO(n)$ are bounded in $\|\cdot\|_{L^\infty(\partial_+SM)}$-norm by some fixed constant $B=B(n)$. Thus
\begin{align*}
&E^1_{\Phi_0}\big[ \log \frac{p_\Phi}{p_{\Phi_0}}(Y, (X,V))\big]^2 \\
& \le 2E_{\lambda}^1 \Big[\sum_{j,k}\frac{1}{2}(C_\Phi((X,V))_{j,k}-C_{\Phi_0}((X,V))_{j,k})^2 \Big]^2 \\
& \quad+ 2E^1_\lambda E_{\varepsilon}^1 \Big[\sum_{j,k}\varepsilon_{j,k} (C_\Phi((X,V))_{j,k}-C_{\Phi_0}((X,V))_{j,k}) \Big]^2 \\
& \le c'(B,n) \|C_\Phi - C_{\Phi_0}\|_{L^2}^2 \le c(n)C_1 \|\Phi-\Phi_0\|_{L^2}^2
\end{align*}
for some constant $c(n)>0$, where we have also used that $\varepsilon_{j,k} \sim^{i.i.d.}N(0,1)$ implies, for $(x,v) \in \partial_+ SM$ fixed,
$$\sum_{j,k}\varepsilon_{j,k} \big(C_\Phi((x,v))_{j,k}-C_{\Phi_0}((x,v))_{j,k}\big) \sim N(0, \frob{C_\Phi(x, v)-C_{\Phi_0}(x,v)}^2),$$ 
and again (\ref{eq:cont0}), so that the overall result follows from appropriate choice of $\kappa>0$
\end{proof}
 
 We now turn to lower bound the small ball probabilities $\Pi(\mathcal B_N(\kappa))$ for the prior $\Pi$ featuring in Theorem \ref{main}  where for the given $\alpha$ we will choose
 \begin{equation}\label{choices}
 \delta_N = N^{-\alpha/(2\alpha+2)} \text{ so that }  \sqrt N \delta_N = N^{1/(2\alpha +2)}.
 \end{equation}
Note that $\sqrt N \delta_N$ precisely equals the rescaling of the prior in (\ref{tampering}). Let us recall the base RKHS $\mathcal H$ from Condition \ref{pco}.
\begin{Lemma}\label{smallv}
Let $\Pi=\times_{j=1}^{\bar n} \Pi_B$ be the prior for $\Phi$ from Theorem \ref{main} with $\alpha>\beta+1, \beta>0$, assume $\Phi_0 \in \mathcal H$ and choose $\delta_N$ as in (\ref{choices}). Let $\mathcal B_N(\kappa)$ be as in Lemma \ref{klrid}. Then for every $\kappa>0$ there exists a constant $C'=C'(\kappa, \alpha, \|\Phi_0\|_\mathcal H, n, M)$ such that for every $N \in \mathbb N$,
$$\Pi(\mathcal B_N(\kappa)) \ge  \exp\{-C' N \delta_N^2 \}.$$ In particular, for $B_N$ as in (\ref{klnbhd}) in Theorem \ref{general}, there exists a finite  constant $$C=C(\alpha, \|\Phi_0\|_\mathcal H, n, M)>0$$ such that for every $N \in \mathbb N$,
\begin{equation} \label{verif}
\Pi(B_N) \ge  \exp\{-C N \delta_N^2 \}.
\end{equation}
\end{Lemma}
\begin{proof}
Since $\|\Phi - \Phi_0\|_{L^2(M)} \le \bar n\max_j \|B_j-B_{0,j}\|_{L^2(M)}$, to prove the first inequality it suffices to lower bound, by independence of the $B_j$'s, 
$$\prod_{j=1}^{\bar n}\Pi_B\big(B: \|B-B_{0,j}\|_{L^2(M)} \le \delta_N /(\kappa\bar n) \big), \quad \bar n = \dim(\so(n)).$$ The sets $\{b:\|b\|_{L^2(M)} \le c\}, c>0,$ are convex and symmetric, hence by \cite[Corollary 2.6.18]{GN16} we have for every $j$ fixed,
\begin{align*}
\Pi_B (\|B-B_{0,j}\|_{L^2(M)} \le \delta_N /(\kappa \bar n) ) &\ge e^{-\|B_{0,j}\|^2_{RKHS(\Pi_B)}/2} \Pi_B(\|B\|_{L^2(M)} \le \delta_N /(\kappa\bar n)) \\
&=e^{-N\delta_N^2 \|B_{0,j}\|^2_{\mathcal H}/2} \Pi_B(\|B\|_{L^2(M)}\le \delta_N/(\kappa\bar n))
\end{align*}
where we have used that $$\|B_{0,j}\|^2_{RKHS(\Pi_B)}=N \delta_N^2\|B_{0,j}\|^2_{\mathcal H}<\infty$$ in view of (\ref{tampering}), (\ref{choices}), (and where we refer to \cite[Exercise 2.6.5]{GN16} or \cite[Lemma I.16]{GvdV17} for standard preservation properties of RKHS under linear transformations). 

We next bound the centred probability which by (\ref{tampering}), (\ref{choices}) equals $$\Pi_B(\|B\|_{L^2(M)} \le \delta_N /(\kappa\bar n)) = \Pi'(\|f'\|_{L^2(M)} \le \sqrt N \delta^2_N/(\kappa \bar n)).$$ By Condition \ref{pco} the RKHS of the Gaussian law of $f'$ in $C(M)$ is continuously imbedded into $H^\alpha(M)$. The unit ball $U$ of this space satisfies the bound
\begin{equation}\label{metsob}
\log N(U, \|\cdot\|_{L^2(M)}, \epsilon) \le (A/\epsilon)^{2/\alpha},~~0<\epsilon<A,~\text{for some }A>0, 
\end{equation}
for its $L^2(M)$-covering numbers: Indeed, since the simple surface $M$ is diffeo-morphic to a disk, we can extend all functions $f$ in $H^\alpha(M)$ to elements $f_e$ of the Sobolev space $H^\alpha(I_2)$ on the $2$-torus $I_2 =(0,1]^2 \supset M$, with Sobolev-norm increased by at most a fixed multiplicative constant (Ch.4~in \cite{T}). An appropriate bound for the $L^2(I_2)$-covering numbers of $\{f_e: f \in U\}$ is then provided in \cite[(4.184)]{GN16}, which in turn (since $\|f-f'\|_{L^2(M)} \le \|f_e-f_e'\|_{L^2(I_2)}$ for all $f,f' \in L^2(M)$) also bounds the $L^2(M)$-covering numbers of $U$ as required.

To proceed, we can now use (\ref{choices}) and \cite[Theorem 1.2]{LL99} (with the value of $\alpha$ there equal to our $2/\alpha$) to lower bound the last small ball probability by 
\begin{equation}\label{fit}
\exp\big\{ -c\big[\sqrt N \delta^2_N/(\kappa \bar n)\big]^{-\frac{4/\alpha}{2-(2/\alpha)}}\big\} \ge \exp\{-c_0 N\delta_N^2\},~~\text{ noting } \sqrt N \delta_N^2 = N^{-(\alpha-1)/(2\alpha+2)} 
\end{equation}
for constants $c=c(\alpha), c_0=c_0(\kappa, n, \alpha)$ and since $\alpha>1$. Combining what precedes proves the first inequality of the lemma with
\begin{equation}
C'=\frac{1}{2}\sum_{j=1}^{\bar n}\|B_{0,j}\|^2_{\mathcal H} + \bar n c_0
\end{equation}
The second inequality (\ref{verif}) now follows from the first and Lemma \ref{klrid}.
\end{proof}

We note that the proof in fact shows that the constant $C$ depends only on \textit{upper bounds} for $\|\Phi_0\|_\mathcal H$.

 \subsubsection{Excess mass and complexity condition}
 
 Having determined the constant $C$ in (\ref{smallball}) for the Gaussian prior in Theorem \ref{main}, we now turn to verifying the remaining conditions (\ref{excessmass}) and (\ref{ment}) in Theorem \ref{general} for a suitable choice of $\mathcal F_N$ that will provide sufficient regularity of the posterior distribution to combine it with our stability estimates for the map $\Phi \mapsto C_\Phi$. 
 \begin{Lemma}\label{regularity}
Let $\Pi$ be the prior from Theorem \ref{main} with $\alpha>\beta+1, \beta>0$, let $\delta_N$ be as in (\ref{choices}) and assume $N \delta^2_N \ge 1$. For $m>0$ define subsets of $C(M, \so(n))$ as 
\begin{align*}
\mathcal F_N &=\Big\{\Phi: \Phi=\Phi_1 + \Phi_2,   \|\Phi_1\|_{L^2} \le \delta_N, \|\Phi_2\|_{H^\alpha} \le m, \|\Phi\|_{C^\beta} \le m  \Big\}
\end{align*} 
a) Then for every $K>0$ we can choose $m$ large enough such that $$\Pi(\mathcal F_N) \ge 1- e^{-KN \delta_N^2}.$$ b) Moreover for some $c=(m, \alpha,n, vol(M))$ we have $$\log N(\mathcal F_N, h, \delta_N)  \le cN \delta_N^2.$$ 
\end{Lemma}
\begin{proof}
a) Recalling (\ref{tampering}), (\ref{choices}), we can identify a prior draw $\Phi$ with the vector field $$(B_1, \dots, B_{\bar n}) = \frac{1}{\sqrt{N} \delta_N} (f'_1, \dots, f'_{\bar n}),~~f'_{j} \sim^{i.i.d.} \Pi'.$$ We denote by $\Pi'_{\bar n}$ the product measure describing the law of the centred Gaussian random variable $(f'_1, \dots, f'_{\bar n})$ in the Banach space $\times_{j=1}^{\bar n}C(M)$. 

Write next $\mathcal F_N=\mathcal F_{N,1} \cap \mathcal F_{N,2}$ where, with $f'_{i,\cdot}$ corresponding to $\Phi_i$, $i=1,2,$ 
\begin{align*}
    \mathcal F_{N,1} &= \Big\{(f'_j=f'_{1,j}+f'_{2,j})_{j=1}^{\bar n}:\sum_{j=1}^{\bar n}\|f'_{1,j}\|^2_{L^2} \le  N \delta_N^4, \sum_{j=1}^{\bar n} \|f'_{2,j}\|^2_{H^\alpha(M)} \le m^2 N \delta^2_N\Big\}, \\
    \mathcal F_{N,2} &= \Big\{(f'_1, \dots, f'_{\bar n}): \sum_{j=1}^{\bar n} \|f'_j\|_{C^\beta(M)} \le m \sqrt N \delta_N  \Big\},    
\end{align*}
so that it suffices to bound the prior probabilities of the complements of $\mathcal F_{N,1}, \mathcal F_{N,2}$.

We first turn to $\mathcal F_{N,2}$. By Condition \ref{pco} the vector field $(f'_1, \dots, f'_{\bar n})$ defines a Gaussian Borel random variable in a separable linear subspace $\mathcal S$ of $\times_{j=1}^{\bar n} C^\beta(M)$. By the Hahn-Banach theorem its $\times_{j=1}^{\bar n} C^\beta(M)$-norm can then be represented as a countable supremum $$\|(f'_1, \dots, f'_{\bar n})\|_{\times_{j=1}^{\bar n} C^\beta(M)} = \sup_{t \in T} |t(f'_1, \dots, f'_{\bar n})|$$ of bounded linear real functionals $T=(t_m: m \in \mathbb N)$ defined on $(\mathcal S, \|\cdot\|_{\times_{j=1}^{\bar n} C^\beta(M)})$. We then apply a version of Fernique's theorem \cite{F75}, concretely \cite[Theorem 2.1.20]{GN16}, to the centred Gaussian process $(X(t):=t(f'_1, \dots, f'_{\bar n}): t \in T)$ to deduce that for some fixed constant $D>0$, $$E\sum_{j=1}^{\bar n} \|f'_j\|_{C^\beta(M)}  \le D<\infty,$$ and then also, for $m=m(D)$ large enough and since $N \delta_N^2 \ge 1$,
\begin{align*}\Pi(\mathcal F_{N,2}^c) &\le \bar \Pi^{\bar n} \Big(\sum_{j=1}^{\bar n} \|f'_j\|_{C^\beta(M)}   - E\sum_{j=1}^{\bar n} \|f'_j\|_{C^\beta(M)}  \ge m \sqrt N \delta_N/2\Big) \le 2e^{-km^2 N \delta_N^2}
\end{align*} 
for $k$ a fixed constant, which can be made less than $e^{-KN \delta_N^2}/2$ for any $K$ provided $m=m(K, k, D)$ is chosen large enough. 

It remains to show that $\Pi(\mathcal F_{N,1}) \ge 1-\frac{1}{2}\exp\{-KN \delta_N^2\}$ for $m$ large enough. Using the continuous imbedding $\mathcal H \subset H^\alpha(M)$ with imbedding constant $c'$ (cf.~Condition \ref{pco}), it suffices to lower bound
\begin{align*}
&\Pi'_{\bar n}\Big((f'_j=f'_{1,j}+f'_{2,j})_{j=1}^{\bar n}:\sum_{j=1}^{\bar n}\|f'_{1,j}\|_{L^2}^2 \le  N \delta_N^4,  \Big(\sum_{j=1}^{\bar n} \|f'_{2,j}\|^2_{\mathcal H}\Big)^{1/2} \le \frac{m}{c'} \sqrt N \delta_N \Big) \notag \\
\quad \quad &=\Pi'_{\bar n}\big(\bar A_N+ m_N O_{\mathcal H} \big) \end{align*}
where $O_{\mathcal H}$ is the unit ball in $\times_{j=1}^{\bar n} \mathcal H$ and where we define $$\bar A_N \equiv \Big\{\omega \in \times_{j=1}^{\bar n}C(M):\|\omega\|_{L^2} \le  \sqrt N \delta_N^2\Big\},\qquad m_N \equiv \frac{m\sqrt N \delta_N}{c'}.$$ By Borell's \cite{B75} isoperimetric inequality (see \cite[Theorem 2.6.12]{GN16}) the last probability is bounded below by 
\begin{equation} \label{isop}
\Phi\big(\Phi^{-1}(\Pi'_{\bar n}(\bar A_N))+m_N\big)
\end{equation}
where $\Phi=\Pr(Z \le \cdot)$ is the cumulative distribution function of a $N(0,1)$ random variable $Z$. By the same arguments as those leading to (\ref{fit}) above, we have $$\Pi'_{\bar n}(\bar A_N) \ge \exp\{-c^2_2 N \delta_N^2\}~\text{ for } c_2=c_2(n, \alpha)>0,$$ and using the basic inequality $\Phi^{-1}(u) \ge - \sqrt{2 \log_- u}, 0<u<1,$ (see \cite[Lemma K.6]{GvdV17}) and monotonicity of $\Phi$ we can further lower bound (\ref{isop}) by
$$\Phi\big(\big(-c_2 \sqrt{2} +\frac{m}{c'}\big) \sqrt N \delta_N\big).$$
Now given $K$ define $$m_N' \equiv - \Phi^{-1}\big[\exp(-KN \delta_N^2)/2 \big]$$ which by the previous inequality for $\Phi^{-1}$ can be made less than or equal to $(\frac{m}{c'}-c_2\sqrt 2) \sqrt N \delta_N$ whenever $m=m(K, c_2, c')$ is large enough. Conclude  that the penultimate display is lower bounded by 
\begin{align*}
    \Phi\big(- \Phi^{-1}\big[\exp(-KN \delta_N^2)/2 \big]\big) &= 1-\Phi\big(\Phi^{-1}\big[\exp(-KN \delta_N^2)/2 \big]\big) \\
    &= 1- \frac{1}{2}\exp\{-KN \delta_N^2\},    
\end{align*}
completing the proof of Part a).

b) To prove Part b), note first that to construct a $\delta_N$-covering of $\mathcal F_N$ in $\|\cdot\|_{L^2(M)}$-distance it suffices, by definition of $\mathcal F_N$, to construct such a covering for a $H^\alpha(M)$-ball of radius $m$, so that (\ref{metsob}) and the definition of $\delta_N$ give (with $A'>0$)
\begin{equation} \label{covdel}
\log N(\mathcal F_N, \|\cdot\|_{L^2(M)}, \delta_N) \le (A'/\delta_N)^{2/\alpha} \le b N \delta_N^2 \text{ for some }b=b(m, \alpha, \bar n)>0.
\end{equation}
Lemma \ref{birge} and (\ref{eq:cont0}) imply that such a covering induces a $(C_1\sqrt {c_1})\delta_N$-covering of $\mathcal F_N$ in the Hellinger distance $h$ of log-cardinality at most $b N \delta_N^2$. Since $\|\cdot\|_{L^2(M)}$ is a norm and hence homogeneous, we can increase the constant from $b$ to $c=c(b, c_1,C_1, \alpha)$ in (\ref{covdel}) and obtain a $\delta_n$-covering for $h$. The desired inequality in Part b) follows. 
\end{proof}

\begin{Remark}\normalfont 
We note that the introduction of the set $\mathcal F_{N,1}$ and the use of Borell's inequality in the previous Lemma could be avoided if one wishes to prove Theorem \ref{main} only for any $\eta>0$ (in this case a minor adaptation of Theorem \ref{general} and of (\ref{covdel}) can be shown to give a slightly worse rate $\delta_N'=N^{-\beta/(2\alpha+2)}$ in (\ref{ll}) below). We give this argument however to obtain our sharper bound for $\eta$ in (\ref{eta}).
\end{Remark}

\subsubsection{Final contraction theorem}

We now put everything together to establish a posterior contraction theorem for $\Phi$ and subsequently deduce Theorem \ref{main}  from it.

\begin{Theorem}\label{overall}
Under the hypotheses of Theorem \ref{main}, with $\alpha>\beta+1, \beta>0, \delta_N=N^{-\alpha/(2\alpha+2)}$ and $C$ from (\ref{verif}), we have for all $m'$ large enough that 
\begin{equation}\label{ll}
P_{\Phi_0}^N \Big( \Pi\big(\Phi: \|C_\Phi - C_{\Phi_0}\|_{L^2(\partial_+SM)} \le m'\delta_N, \|\Phi\|_{C^{\beta}(M)} \le m'|D_N) \ge 1- e^{-(C+3)N \delta_N^2} \Big) \to 1
\end{equation}
 as $N \to \infty$. Moreover, if $\beta>2$ then we have for every integer $\bar \beta$ such that $1<\bar \beta<\beta$ and all $m''$ large enough,
$$P_{\Phi_0}^N \Big( \Pi\big(\Phi: \|\Phi - \Phi_0\|_{L^2(M)} \ge m'' \delta_N^{(\bar \beta-1)/\bar \beta}|D_N\big) \ge e^{-(C+3)N \delta_N^2} \Big) \to_{N \to \infty} 0.$$ 
\end{Theorem}
\begin{Remark}\label{beta} \normalfont
The constraint $\beta>2$ in the second limit in Theorem \ref{overall} is only required to allow space for an \textit{integer} $\bar \beta \in (1,\beta)$ in the following proof, when combining the interpolation inequality (\ref{police2}) with Theorem \ref{thm:forward} for $k = \bar \beta \in \mathbb N$. If a version of Theorem \ref{thm:forward} were established for non-integer $k$ then $\beta>1$ and real $\bar \beta \in (1,\beta)$ would be permitted in Theorem \ref{overall} (and then also in Theorem \ref{main}).
\end{Remark}
\begin{proof}
From Lemmata \ref{smallv} and \ref{regularity} with $K=2C+6$ and Theorem \ref{general} we deduce for $m=m(C)$ large enough, and as $N \to \infty$
\begin{equation*}
P_{\Phi_0}^N \left(\Pi\big(\Phi: \{\Phi: h(p_\Phi, p_{\Phi_0}) \le m \delta_N\} \cap \{\|\Phi\|_{C^{\beta}(M)} \le m\} | D_N\big) \le 1- e^{-(C+3)N \delta_N^2} \right) \to 0.
\end{equation*}
 Applying Lemma \ref{birge} gives the first limit (\ref{ll}) with $m' =(1+\sqrt{c_0}) m$.

\smallskip

To prove the second limit we will apply the stability estimate Theorem \ref{thm:stability} in the form (\ref{stabi}) with $\Psi =\Phi_0$. By hypothesis we have $\|\Phi_0\|_{C^1(M)} \lesssim \|\Phi_0\|_{C^\alpha(M)}<\infty$; as a consequence for all $\Phi$ contained in the event in (\ref{ll}) with $\beta>2$, the constants $c(\Phi, \Phi_0)$ from (\ref{eq:cPhiPsi}) are uniformly bounded by a fixed constant that depends on $m', \|\Phi_0\|_{C^1(M)}$ and hence for those $\Phi$'s
\begin{equation}\label{stab0}
\|\Phi - \Phi_0\|_{L^2(M)} \le D(\|\Phi_0\|_{C^1(M)}, m') \|C_{\Phi} - C_{\Phi_0}\|_{H^{1}(\partial_+ SM)}.
\end{equation}
To proceed we will need a standard interpolation result for Sobolev spaces on the manifold $\partial_+ SM$ to the effect that
\begin{equation}\label{police2}
\|W\|_{H^{1}(\partial_+ SM)} \lesssim \|W\|_{L^2(\partial_+ SM)}^{(k-1)/k} \|W\|_{H^{k}(\partial _+ SM)}^{1/k}
\end{equation}
for all $W \in H^{k}(\partial_+ SM)$ and any $k > 1$. [For real-valued functions this can be proved using standard arguments from Ch.4~in \cite{T} and these results extend to matrix-fields in a straightforward way.] Moreover we will use the basic inequality
\begin{equation}
\|\Phi\|_{H^{\bar \beta}(M)} \lesssim \|\Phi\|_{C^{\bar \beta}(M)} \le \|\Phi\|_{C^\beta(M)},
\end{equation}
for all $\Phi \in C^\beta(M)$. Now Theorem \ref{thm:forward} implies that for all $\Phi$'s in the event in (\ref{ll}) the corresponding $\|C_\Phi\|_{H^{\bar \beta}(\partial_+ SM)}$'s are uniformly bounded by a fixed constant that depends on $m', \beta, \bar \beta$ only. Likewise $$\|C_{\Phi_0}\|_{H^{\bar \beta}(\partial_+ SM)} \le\|C_{\Phi_0}\|_{C^{\alpha}(\partial_+ SM)} \lesssim (1+\|\Phi_0\|_{C^\alpha})<\infty$$ in view of Theorem \ref{thm:forward} and since $\Phi_0 \in C^\alpha$ for $\alpha>\bar \beta$ by hypothesis. Hence for such $\Phi$'s the combination of (\ref{stab0}) and (\ref{police2}) with $W=C_\Phi-C_{\Phi_0}, k=\bar \beta$ gives
\begin{equation*}
\|\Phi - \Phi_0\|_{L^2(M)} \lesssim  \|C_{\Phi} - C_{\Phi_0}\|_{L^2(\partial_+ SM)}^{(\bar \beta-1)/\bar \beta} \|C_{\Phi} - C_{\Phi_0}\|^{1/\bar \beta}_{H^{\bar \beta}(\partial _+ SM)}  \lesssim \delta_N^{(\bar \beta-1)/\bar \beta}.
 \end{equation*}
The second conclusion of Theorem \ref{overall} now follows from the preceding inequalities and (\ref{ll}).

\subsubsection{Completion of the proof of Theorem \ref{main}} \label{uiq}

The last step is to show that the posterior contraction rate in the second limit of Theorem \ref{overall} carries over to the posterior mean $E^\Pi[\Phi|D_N]$.  For any integer $\bar \beta \in (1, \beta)$ and every
\begin{equation} \label{eta}
0<\eta<\frac{\alpha}{2\alpha +2} \times \frac{\bar \beta-1}{\bar \beta},
\end{equation}
we have as $N \to \infty$ $$\eta_N := m''\delta_N^{(\bar \beta-1)/\bar \beta} \simeq N^{-\frac{\alpha}{2\alpha+2} \frac{\bar \beta-1}{\bar \beta}}=o(N^{-\eta}).$$  Then by the inequalities of Jensen and Cauchy-Schwarz
\begin{align*}\label{deco}
&\|E^\Pi[\Phi|D_N]-\Phi_0\|_{L^2(M)} \notag \\
&\le E^\Pi [\|\Phi-\Phi_0\|_{L^2(M)}|D_N] \\
&\le \eta_N + E^\Pi [\|\Phi-\Phi_0\|_{L^2(M)}1\{\|\Phi-\Phi_0\|_{L^2(M)} \ge \eta_N\}|D_N] \notag \\
&\le \eta_N + [E^\Pi [\|\Phi-\Phi_0\|_{L^2(M)}^2|D_N]^{1/2} \Pi(\|\Phi-\Phi_0\|_{L^2(M)} \ge \eta_N|D_N)^{1/2}
\end{align*} 
and it suffices to show that the second summand is stochastically $O(\eta_N)$ as $N \to \infty$. 

Arguing as in the proof of Theorem \ref{general} and using Lemma \ref{smallv} implies that the sets $A_N$ from (\ref{anset}) with $C$ from (\ref{verif}) satisfy $P_{\Phi_0}^N(A_N) \to 1$ as $N \to \infty$. Now Theorem \ref{overall}, (\ref{post}) and Markov's inequality imply 
\begin{align*}
&P^N_{\Phi_0} \Big(E^\Pi [\|\Phi-\Phi_0\|_{L^2(M)}^2|D_N] \times \Pi(\|\Phi-\Phi_0\|_{L^2(M)} \ge \eta_N|D_N) > \eta_N^2 \Big)\\
& \le P^N_{\Phi_0} \Big(E^\Pi [\|\Phi-\Phi_0\|_{L^2(M)}^2|D_N] e^{-(C+3)N \delta_N^2} > \eta_N^2 \Big) + o(1) \\
& \le P^N_{\Phi_0} \Big(e^{-(C+3)N \delta_N^2} \frac{\int \|\Phi-\Phi_0\|_{L^2(M)}^2 \prod_{i=1}^N \frac{p_\Phi}{p_{\Phi_0}}(Y_i, (X_i, V_i))d\Pi(\Phi) }{\int  \prod_{i=1}^N \frac{p_\Phi}{p_{\Phi_0}}(Y_i, (X_i, V_i)) d\Pi(\Phi) } > \eta_N^2, A_N \Big) + o(1)  \\
& \le e^{-N \delta_N^2} \eta_N^{-2} E^N_{\Phi_0}\int \|\Phi-\Phi_0\|_{L^2(M)}^2 \prod_{i=1}^N \frac{p_\Phi}{p_{\Phi_0}}(Y_i, (X_i, V_i))d\Pi(\Phi) \\
&\le e^{-N\delta_N^2}\eta_N^{-2} \int \| \Phi - \Phi_0\|_{L^2(M)}^2 d\Pi(\Phi) \lesssim e^{- N\delta_N^2} \eta_N^{-2} \to_{N \to \infty} 0
\end{align*}
where we have also used Fubini's theorem, (\ref{com}), and that the Gaussian measure $\Pi$ is supported in $L^2(M)$ and hence integrates $\|\Phi\|^2_{L^2}$ to a finite constant (see, e.g., \cite[Exercise 2.1.5]{GN16}). 
\end{proof}

        





\ack 

We would like to thank the referee for helpful remarks and suggestions. We are further very grateful to Bill Lionheart for having introduced us to polarimetric neutron tomography and its connection to the non-abelian X-ray transform. 
FM was supported by NSF grant DMS-1814104 and a UC Hellman Fellowship. RN was supported by the European Research Council under ERC grant No. 647812 (UQMSI).  GPP thanks the University of California at Santa Cruz and the University of Washington for hospitality while this work was in progress. GPP was supported by the Leverhulme trust and EPSRC grant EP/R001898/1.


\frenchspacing
\bibliographystyle{cpam}

\end{document}